\documentclass[11pt,a4paper,twocolumn]{article}

\usepackage[utf8]{inputenc}
\usepackage[english]{babel}
\usepackage{amsmath,amssymb,amsthm,mathtools}
\usepackage{microtype}
\microtypesetup{expansion=false}
\usepackage{physics}
\usepackage{geometry}
\usepackage{hyperref}
\usepackage{cite}
\usepackage{abstract}
\usepackage{xcolor}
\usepackage{booktabs,tabularx,array}
\usepackage{enumitem}

\hypersetup{
    colorlinks=true,
    linkcolor=blue,
    filecolor=magenta,
    urlcolor=cyan,
    citecolor=red,
}

\newtheorem{theorem}{Theorem}
\newtheorem{proposition}[theorem]{Proposition}
\newtheorem{lemma}[theorem]{Lemma}

\newtheorem{corollary}[theorem]{Corollary}

\newtheorem{example}[theorem]{Example}

\newcommand{\HH}{\mathbb{H}}
\newcommand{\CC}{\mathbb{C}}
\newcommand{\RR}{\mathbb{R}}
\newcommand{\OO}{\mathbb{O}}
\newcommand{\HP}{\mathbb H\mathrm P}
\newcommand{\OP}{\mathbb O\mathrm P}
\newcommand{\Cl}{\mathrm{Cl}}
\newcommand{\RePart}{\operatorname{Re}}
\newcommand{\Rcomp}{\mathbin{\circledcirc}}
\newcommand{\CPM}{\operatorname{CPM}}
\newcommand{\CPstar}{\operatorname{CP}^{*}}
\newcommand{\Sp}{\mathrm{Sp}}

\newcommand{\Imag}{\operatorname{Im}}

\newcommand{\id}{\mathrm{id}}

\newcommand{\End}{\operatorname{End}}
\DeclareMathOperator{\Span}{span}

\title{\Large\textbf{Operational Foundations for Quaternionic and Octonionic Quantum Models}\\[1mm]
\normalsize Exact quaternionic channels and error correction, para-linear operators,\\
and categorical closure boundaries beyond associativity}
\author{\textbf{Santiago Pineda Montoya\textsuperscript{1} and Johan H. R\'ua Mu\~noz\textsuperscript{2}}\\
\textsuperscript{1}\textit{Institute of Mathematics, National University of Colombia}\\
\textsuperscript{2}\textit{Institute of Physics, University of Antioquia}}
\date{}

\begin{document}
\emergencystretch=3em
\raggedbottom

\twocolumn[
  \begin{@twocolumnfalse}
    \maketitle
    \begin{abstract}
\small\sloppy
A scalar field alone does not determine a quantum theory: states, effects, processes, symmetries,
composition, and discard are equally structural.  Realification illustrates this point: an orthogonal
complex structure $\mathsf J^2=-I$ selects the physical real operators and the balanced composite.
Quaternionic quantum mechanics has an analogous exact representation on a doubled complex space
selected by an antiunitary symplectic structure $\Theta^2=-I$.  Within that sector, a Choi fixed-point
condition characterizes when a complex channel admits quaternionic Kraus operators, while exact
correction of a finite-dimensional right-quaternionic code is characterized by compression
coefficients in the real center and admits an explicit recovery.  For octonions, nonassociativity precludes a unique continuation; para-linear,
categorical, sectorial, Jordan, Clifford-envelope, and Moufang models are compared by the
operational structures they retain and the additional data required to define composites, channels,
or recovery.  Across these cases, an exact ambient representation does not erase the complex or
symplectic structure that selects the physical theory.
\end{abstract}
  \end{@twocolumnfalse}
]

\section{Introduction and scope}\label{sec:introduction}

Quantum theory over the quaternions $\HH$ is a natural associative extension of the usual complex
Hilbert-space formalism.  Finite-dimensional amplitudes form right quaternionic Hilbert modules,
observables are quaternionic-Hermitian matrices, and reversible transformations belong to the
compact symplectic group.  Noncommutativity nevertheless makes several conventions essential: the
side on which scalars act, the linear argument of the inner product, the operational trace, and the
class of admissible maps must all be fixed explicitly.  Quaternionic quantum mechanics,
generalized measurements, completely positive maps, and exact complex simulation have substantial
prior literatures \cite{Adler1995,Graydon2013,Asorey2007,BhardwajSharmaKaushik2021}.  The
basis-independent real trace and the quaternionic form of Gleason's theorem were clarified by
Moretti and Oppio, while symmetry-based reductions show how a preferred complex structure can arise
inside a quaternionic Hilbert space \cite{MorettiOppioGleason2018,MorettiOppioPoincare2019,
Gantner2018}.  Ordered circuits and non-naive composition schemes have also been investigated
\cite{Razon1991,Fernandez2004,Barnum2020}.

Computational and geometric analyses provide complementary motivation for the operational questions
considered here.  Quaternionic and octonionic circuit models exhibit polynomial equivalence between
the quaternionic and standard complex computational models, together with an obstruction caused by
octonionic nonassociativity \cite{RuaMahechaPineda2025}.  Division algebras also organize Clifford
structures, projective state spaces, and Hopf fibrations, while clarifying the difference between
hypercomplex single-system kinematics and complex multipartite entanglement
\cite{RuaPineda2026Geometry}.  The remaining operational layer concerns density operators,
measurements, channels, liftability, exact error correction, and the model-dependent meaning of
octonionic dynamics.

The renewed discussion of real-number quantum mechanics makes the distinction between scalar
coordinates and operational structure especially sharp.  Real formulations with a distinguished
orthogonal complex structure go back at least to Stueckelberg \cite{Stueckelberg1960}.  Renou
\textit{et al.} later showed that ordinary real-amplitude quantum theory, equipped with the standard
real tensor product and unrestricted real observables, differs from complex quantum theory in
independent-source networks \cite{Renou2021}.  Barrios Hita \textit{et al.} study a different
operational theory: a real flag represents multiplication by $i$, physical maps respect the
associated phase gauge, and a quotient composition rule replaces the unrestricted real tensor
product; the resulting model reproduces all multipartite complex predictions
\cite{BarriosHita2026}.  Bang, Cho, and Baek identify the retained operator as a distinguished
complex structure $\mathsf J$ and the quotient as a balanced tensor product, while Maioli, Curado,
and Gazeau give an independent real symplectic formulation with a modified composition rule
\cite{BangChoBaek2026,MaioliCuradoGazeau2026}.  Tiwari's Comment emphasizes the corresponding
interpretive point: eliminating complex coordinates need not eliminate complex operational
structure \cite{Tiwari2026}.  These observations are compatible with the exact equivalence result;
they clarify which real theory is being compared with the complex one.

A separate use of real Hilbert spaces occurs in quantum gravity.  Harlow and Numasawa argue that
gauging spacetime inversions makes the closed-universe quantum-gravity state space real, and Harlow,
Usatyuk, and Zhao combine related arguments with evidence for a one-dimensional global state space
and a large effective Hilbert space for internal observers
\cite{HarlowNumasawa2026,HarlowUsatyukZhao2026}.  Those claims concern physical selection by
gravity and symmetry, rather than the coordinate realification of ordinary complex quantum
mechanics.

The common organizing principle is therefore structural rather than scalar: a quantum model is
specified not only by a base algebra, but also by its state and effect cones, admissible process
algebra, symmetries or superselection rules, composition, and discard.  The real control case embeds
complex quantum mechanics in a real ambient theory selected by $\mathsf J^2=-I$.  Quaternionic
quantum mechanics embeds in a doubled complex ambient theory selected by an antiunitary
$\Theta^2=-I$.  Both encodings preserve all probabilities because the selecting structure is
retained.  Octonions mark a different boundary: nonassociativity obstructs a single ordinary
operator product, so the operational model must be chosen before states or channels can be compared.

Two finite-dimensional quaternionic statements are developed within this framework.  Quaternionic
completely positive maps already admit intrinsic Choi--Kraus descriptions, and their complex
projections and exact simulations have been studied \cite{Asorey2007,Graydon2013,
BhardwajSharmaKaushik2021}.  Fixed-Choi criteria for real channels and channel imaginarity provide
the $\Theta^2=+I$ antecedent \cite{HickeyGour2018,ChenLei2025}.  The channel criterion below specializes this fixed-real-form mechanism to the doubled-complex
symplectic sector $\Theta^2=-I$, yielding a membership criterion and an explicit reconstruction of
quaternionic Kraus operators.  Error correction over real and
quaternionic state spaces has also been considered previously, including an early Euclidean-qubit
scheme and a recent concrete quaternionic code construction \cite{Vlasov2001,
Nyirahafashimana2025}.  Very recent work also uses a $\Theta^2=-I$ time-reversal symmetry to derive
parity selection rules for complex Kramers codes, making specified families of Knill--Laflamme
conditions automatic \cite{KubischtaTeixeira2026}.  Against the established complex and
operator-algebraic Knill--Laflamme frameworks
\cite{KnillLaflamme1997,KribsLaflammePoulinLesosky2006}, the error-correction
result identifies what the finite-dimensional right-quaternionic specialization requires: real
central compression coefficients and a right-$\HH$-linear recovery, together with---as in the
complex case---a separate treatment of the one-dimensional code, whose normalized state space is a
single point.  The scope is restricted to these finite-dimensional formulations and their consequences.

The octonionic analysis is deliberately model-specific.  Established approaches include exceptional
Jordan algebras, selected complex geometries, octonionic Hilbert modules with para-linear maps,
cochain-twisted graded categories, associative real or complex operator envelopes, and finite
Moufang loops \cite{GunaydinPironRuegg1978,DeLeoAbdelKhalek1996,DeLeoDirac1996,
GoldstineHorwitz1964,GoldstineHorwitz1966,HuoRen2022,HuoRen2023,HuoRenSabadini2025,
AlbuquerqueMajid1999,KlimMajid2010}.  They do not share one state space, one composite, or one
notion of error correction.  The objective is to identify the operational content of each route,
derive exact transport where it is available, and state what additional data would be required for a
global octonionic channel theory.

Section~\ref{sec:scalar-presentation} develops the real control case.  Sections~\ref{sec:qkinematics}
--\ref{sec:qcomposition} construct the quaternionic state, channel, Choi, error-correction, and
composition-dependent framework.  Section~\ref{sec:oct-fork} then compares the octonionic
alternatives as subsections of one operational decision tree.

\section{Scalar presentation and retained operational structure}
\label{sec:scalar-presentation}

The passage from a complex Hilbert space to its underlying real space, together with an orthogonal
operator squaring to $-I$, is the standard realification
\cite{Stueckelberg1960,MorettiOppioReal2017}.  Recent work makes the composite-system rule the
central issue: Barrios Hita \textit{et al.} use a quotient flag construction, Bang, Cho, and Baek
identify its hidden complex structure and balanced tensor product, and Maioli, Curado, and Gazeau
formulate a related real symplectic composition rule
\cite{BarriosHita2026,BangChoBaek2026,MaioliCuradoGazeau2026}.  The finite-dimensional
calculation supplies a precise control case for the quaternionic and octonionic questions.
Let $\mathcal H_{\CC}=\CC^d$ and let $V:=(\mathcal H_{\CC})_{\RR}$ denote the same additive space regarded
as a real Hilbert space.  The symbols $\Tr_{\CC}$ and $\Tr_{\RR}$ denote the ordinary matrix
traces on complex and real matrices, respectively.  For $\psi=x+iy$, with $x,y\in\RR^d$, define the realification
\begin{equation}
 \mathcal S_d(\psi):=
 \begin{pmatrix}x\\y\end{pmatrix}\in\RR^{2d},
 \qquad
 \mathsf J_d:=
 \begin{pmatrix}0&-I_d\\ I_d&0\end{pmatrix}.
 \label{eq:realification-state-J}
\end{equation}
Then $\mathsf J_d^2=-I_{2d}$,
$\mathcal S_d(i\psi)=\mathsf J_d\mathcal S_d(\psi)$, and
\begin{equation}
 \mathcal S_d(e^{i\alpha}\psi)
 =e^{\alpha\mathsf J_d}\mathcal S_d(\psi).
 \label{eq:realification-phase}
\end{equation}
Thus the complex phase orbit becomes an orbit of the planar rotation group $SO(2)$ generated by
the distinguished real-linear operator $\mathsf J_d$.  If
$g(v,w):=v^{\mathsf T}w$ is the real inner product on $V$, then for
$\psi,\phi\in\mathcal H_{\CC}$ one recovers the full complex inner product from the pair
$(g,\mathsf J_d)$:
\begin{equation}
 \langle\psi,\phi\rangle_{\CC}
 =g\!\left(\mathcal S_d\psi,\mathcal S_d\phi\right)
 -i\,g\!\left(\mathcal S_d\psi,\mathsf J_d\mathcal S_d\phi\right).
 \label{eq:realification-inner-product-reconstruction}
\end{equation}
The real metric alone retains only the real part; the distinguished complex structure restores the
imaginary part and hence the original complex Hilbert-space geometry.

For a complex matrix $A=X+iY\in M_d(\CC)$, with $X,Y\in M_d(\RR)$, define
\begin{equation}
 \mathcal T_d(A):=
 \begin{pmatrix}X&-Y\\Y&X\end{pmatrix}.
 \label{eq:realification-operator}
\end{equation}
A direct block calculation gives
\begin{align}
 \mathcal T_d(A)\mathcal S_d(\psi)
 &=\mathcal S_d(A\psi),\notag\\
 \mathcal T_d(AB)
 &=\mathcal T_d(A)\mathcal T_d(B),\notag\\
 \mathcal T_d(A^\dagger)
 &=\mathcal T_d(A)^{\mathsf T}.
 \label{eq:realification-intertwining}
\end{align}
Moreover, a real-linear operator on $V$ is complex-linear precisely when it commutes with
$\mathsf J_d$.  Hence the physical algebra of the exactly equivalent real presentation is not all
of $M_{2d}(\RR)$ but the commutant
\begin{equation}
 \End_{\CC}(\mathcal H_{\CC})
 \cong
 \{R\in\End_{\RR}(V):R\mathsf J_d=\mathsf J_dR\}.
 \label{eq:J-commutant}
\end{equation}
In particular, if every vector on the orbit in Eq.~\eqref{eq:realification-phase} is to represent the
same physical pure state, each real symmetric effect must commute with $\mathsf J_d$; otherwise the
flag rotation is measurable.  The restriction on effects and state transformations is therefore
part of the operational theory, not a notational afterthought \cite{BarriosHita2026,BangChoBaek2026}.
For a complex density matrix $\rho$ and effect $E$, define the normalized real representatives
\begin{equation}
 \rho_{\RR}:=\frac12\mathcal T_d(\rho),
 \qquad
 E_{\RR}:=\mathcal T_d(E).
 \label{eq:realification-born-representatives}
\end{equation}
Trace doubling then gives the exact Born rule
\begin{equation}
 \Tr_{\RR}(E_{\RR}\rho_{\RR})
 =\Tr_{\CC}(E\rho).
 \label{eq:realification-born}
\end{equation}
The factor $1/2$ compensates for
$\Tr_{\RR}\mathcal T_d(A)=2\RePart\Tr_{\CC}A$.  Likewise, let $A_1,\ldots,A_\ell\in M_d(\CC)$ satisfy
$\sum_{s=1}^{\ell}A_s^\dagger A_s=I_d$, and define the Kraus map
$\Phi(X):=\sum_{s=1}^{\ell}A_sXA_s^\dagger$.  With
$L_s:=\mathcal T_d(A_s)$, one has
\begin{align}
 \Phi_{\RR}(Z)
 &:=\sum_sL_sZL_s^{\mathsf T},\notag\\
 \Phi_{\RR}\!\left(\tfrac12\mathcal T_d(X)\right)
 &=\tfrac12\mathcal T_d(\Phi(X)).
 \label{eq:realification-channel}
\end{align}
Moreover,
$\sum_sL_s^{\mathsf T}L_s
 =\mathcal T_d(\sum_sA_s^\dagger A_s)=I_{2d}$,
so $\Phi_{\RR}$ is trace preserving.  Thus the distinguished $\mathsf J_d$ constrains open
dynamics as well as effects.

The same issue becomes decisive for composition.  If $V_A$ and $V_B$ are the realifications of
complex Hilbert spaces with complex structures $\mathsf J_A$ and $\mathsf J_B$, define
\begin{align}
 r_J(v,w)
 &:=\mathsf J_Av\otimes w-v\otimes\mathsf J_Bw,\notag\\
 \mathcal R_J
 &:=\Span_{\RR}\{r_J(v,w):(v,w)\in V_A\times V_B\},\notag\\
 V_A\boxtimes_J V_B
 &:=(V_A\otimes_{\RR}V_B)/\mathcal R_J.
 \label{eq:J-balanced-tensor}
\end{align}
The quotient imposes the real version of the complex balancing relation
$(iv)\otimes w=v\otimes(iw)$.  Consequently,
\begin{align}
 V_A\boxtimes_JV_B
 &\cong
 (\mathcal H_A\otimes_{\CC}\mathcal H_B)_{\RR},\notag\\
 \dim_{\RR}(V_A\boxtimes_JV_B)
 &=2d_Ad_B.
 \label{eq:J-balanced-isomorphism}
\end{align}
whereas the ordinary real tensor product has dimension $4d_Ad_B$.  The quotient-space flag
construction of Barrios Hita \textit{et al.} realizes this balanced product, and Bang, Cho, and Baek
give the explicit identification with the $\mathsf J$-balanced tensor product
\cite{BarriosHita2026,BangChoBaek2026}.  The symplectic composition of Maioli, Curado, and Gazeau
has the same structural purpose: it retains the real operator that represents multiplication by
$i$ rather than using the unrestricted real Kronecker product \cite{MaioliCuradoGazeau2026}.
Exact multipartite equivalence and retention of complex operational structure are therefore
compatible conclusions.  In network language, factorization with respect to $\boxtimes_J$ is not
the same assumption as factorization in the ordinary real tensor product; reproducing the complex
correlations changes the composite and source-factorization rule rather than invalidating the
ordinary-real foil theory tested by Renou \textit{et al.} \cite{Renou2021}.  A locality condition
on already chosen embeddings does not, by itself, select one composite rule over the other.

The control case provides both a template and a warning.  The operator $\mathsf J$ selects
complex quantum mechanics inside a larger real ambient theory; the antiunitary $\Theta$ introduced
next will select quaternionic quantum mechanics inside a larger complex ambient theory.  In each
case, exact simulation preserves all probabilities because the selecting structure remains part of
the operational specification.  The quaternionic case adds a noncommutative scalar algebra and a
noncanonical subsystem problem, while the octonionic case adds nonassociativity and therefore
requires a choice among genuinely different operator and state-space models.  This hierarchy---$\mathsf J$ inside real coordinates, $\Theta$ inside complex coordinates, and
model-specific closure beyond associativity---does more than organize notation.  It isolates the
question answered by each later construction: which ambient degrees of freedom are redundant, which
symmetry selects the physical sector, and which composition law preserves that sector.  The
quaternionic theorems make these questions exact in an associative noncommutative theory; the
octonionic comparison shows where the same strategy ceases to determine a unique process model.

\section{Quaternionic kinematics, measurements, and exact complex representation}
\label{sec:qkinematics}

This section combines the kinematic and measurement layers because they answer one continuous
question: what is a finite-dimensional quaternionic quantum system, and how are its experimental
probabilities represented exactly inside complex quantum mechanics?  We begin with the intrinsic
objects and introduce the complex representation only after states and measurements are defined.
Throughout, $N\in\mathbb N$ is a positive integer and $I_N$ denotes the identity matrix in
$M_N(\HH)$.  For a quaternion
$q=q_0+q_1i+q_2j+q_3k\in\HH$, with $q_0,q_1,q_2,q_3\in\RR$, define
\begin{align}
 \overline q&:=q_0-q_1i-q_2j-q_3k,\notag\\
 |q|_{\HH}&:=\sqrt{\overline q q}
 =\sqrt{q_0^2+q_1^2+q_2^2+q_3^2}.
 \label{eq:qscalar-norm}
\end{align}
The quantity $|q|_{\HH}$ is the multiplicative scalar norm on the division algebra $\HH$ \cite{Adler1995,Zhang1997}.
It will be kept distinct from the Hilbert norm of a vector, the trace norm of a complex matrix,
and the diamond norm of a channel.

\subsection{Right modules, inner products, and operators}

The amplitude space is
\begin{equation}
 \mathcal H_{\HH}:=\HH^N,
 \label{eq:qhilbert}
\end{equation}
viewed as a \emph{right} module over $\HH$.  Thus, for
$x=(x_1,\ldots,x_N)^{\mathsf T}\in\mathcal H_{\HH}$ and $q\in\HH$, the scalar product
$xq$ is defined componentwise.  For vectors $x,y\in\mathcal H_{\HH}$ and scalars
$a,b\in\HH$, the standard quaternionic inner product is
\begin{equation}
 \langle x,y\rangle_{\HH}:=\sum_{\alpha=1}^{N}\overline{x_\alpha}y_\alpha,
 \qquad
 \langle xa,yb\rangle_{\HH}=\overline a\,\langle x,y\rangle_{\HH}\,b.
 \label{eq:qinner}
\end{equation}
The second identity is the precise meaning of conjugate-linearity in the first argument and
right-linearity in the second.  The order of the factors is part of the definition; it cannot be
changed because quaternionic multiplication is noncommutative.  The inner product also satisfies
\begin{equation}
 \langle y,x\rangle_{\HH}=\overline{\langle x,y\rangle_{\HH}},
 \qquad
 \langle x,x\rangle_{\HH}\in\RR_{\ge0},
 \label{eq:qinner-properties}
\end{equation}
and $\langle x,x\rangle_{\HH}=0$ only when $x=0$.  The induced Hilbert norm is
\begin{equation}
 \|x\|_{\mathcal H_{\HH}}
 :=\sqrt{\langle x,x\rangle_{\HH}}
 =\left(\sum_{\alpha=1}^{N}|x_\alpha|_{\HH}^{2}\right)^{1/2}.
 \label{eq:qvector-norm}
\end{equation}
Thus $|\cdot|_{\HH}$ is used for one quaternion, whereas
$\|\cdot\|_{\mathcal H_{\HH}}$ is used for a vector in $\HH^N$.  For a complex vector
$z=(z_1,\ldots,z_d)^{\mathsf T}\in\CC^d$, we write
\begin{equation}
 \|z\|_{2,\CC}:=\left(\sum_{\alpha=1}^{d}|z_\alpha|^2\right)^{1/2}
 \label{eq:complex-vector-norm}
\end{equation}
for the ordinary complex Euclidean norm.  Channel norms are marked separately by a diamond subscript.

A map $A:\mathcal H_{\HH}\to\mathcal H_{\HH}$ is right-$\HH$-linear when
$A(xq+yr)=A(x)q+A(y)r$ for all $x,y\in\mathcal H_{\HH}$ and $q,r\in\HH$.  Every such map is
represented by a matrix in $M_N(\HH)$ acting from the left, so
\begin{equation}
 \End_{\HH}(\mathcal H_{\HH})\cong M_N(\HH).
 \label{eq:qoperator-algebra}
\end{equation}
For $A\in M_N(\HH)$, the adjoint is $A^\dagger=\overline A^{\,\mathsf T}$ and is characterized
by
\begin{equation}
 \langle Ax,y\rangle_{\HH}=\langle x,A^\dagger y\rangle_{\HH}
 \qquad (x,y\in\mathcal H_{\HH}).
 \label{eq:qadjoint}
\end{equation}
Hermitian matrices $A=A^\dagger$ represent observables.  Reversible transformations are the
quaternionic unitary matrices, collected in the compact symplectic group
\begin{equation}
 \Sp(N):=\{U\in M_N(\HH):U^\dagger U=UU^\dagger=I_N\}.
 \label{eq:spn-definition}
\end{equation}
For a Hermitian matrix $A=A^\dagger$, we write $A\ge0$ when
$\langle x,Ax\rangle_{\HH}\ge0$ for every $x\in\mathcal H_{\HH}$.  For Hermitian
$A$ and $B$, the order relation $A\le B$ means $B-A\ge0$.  A matrix $P\in M_N(\HH)$ is an
\emph{orthogonal projection} when
\begin{equation}
 P=P^\dagger=P^2.
 \label{eq:orthogonal-projection-definition}
\end{equation}
Its range $\operatorname{ran}P$ is a right-$\HH$ submodule, and $P$ acts as the identity on
that range and as zero on its orthogonal complement.  These definitions will be used for density
matrices, effects, and code subspaces below.  No irreversible dynamics is assumed at this stage;
channels and their Kraus operators are introduced only in Section~\ref{sec:channels}.

For an arbitrary matrix $A=(A_{\alpha\beta})\in M_N(\HH)$, define the real trace
\begin{equation}
 \Tr_{\HH}A:=\RePart\sum_{\alpha=1}^{N}A_{\alpha\alpha}.
 \label{eq:real-qtrace}
\end{equation}
If $A$ is Hermitian, its diagonal entries are real and this is simply their sum.  Unlike the
unmodified quaternionic diagonal sum, $\Tr_{\HH}$ is cyclic and therefore gives a real,
basis-independent pairing for states and effects
\cite{Adler1995,Graydon2013,MorettiOppioGleason2018}.

\subsection{Pure states, mixed states, effects, and POVMs}

A unit vector $\psi\in\mathcal H_{\HH}$, meaning
$\|\psi\|_{\mathcal H_{\HH}}=1$ or equivalently
$\langle\psi,\psi\rangle_{\HH}=1$, determines the rank-one projector
\begin{equation}
 \rho_\psi:=\psi\psi^\dagger,
 \qquad
 \rho_\psi x=\psi\langle\psi,x\rangle_{\HH}
 \quad (x\in\mathcal H_{\HH}).
 \label{eq:qpure-state}
\end{equation}
If $u\in\HH$ satisfies $|u|_{\HH}=1$, then $\psi$ and $\psi u$ define the same projector.  Pure
states are therefore quaternionic rays, identified with the projective space $\HP^{N-1}$.

With the operator order defined above, the mixed-state space is
\begin{equation}
 \mathcal D_N(\HH):=
 \{\rho\in M_N(\HH):\rho=\rho^\dagger,\ \rho\ge0,\ \Tr_{\HH}\rho=1\}.
 \label{eq:qstate-space}
\end{equation}
This distinguishes the amplitude space $\mathcal H_{\HH}$ from the physical state space
$\mathcal D_N(\HH)$.  In quaternionic dimension one, every Hermitian matrix is a real
$1\times1$ matrix.  Positivity and unit trace therefore force its sole entry to be $1$, so
\begin{equation}
 \mathcal D_1(\HH)=\{[1]\},
 \label{eq:qstate-space-one-dimensional}
\end{equation}
where $[1]$ denotes the $1\times1$ matrix whose unique entry is the real unit $1$.

An \emph{effect} is a Hermitian matrix $E\in M_N(\HH)$ satisfying
$0\le E\le I_N$.  Let $\Omega$ be a finite set of measurement outcomes.  A
\emph{positive-operator-valued measure} (POVM) is a family
$\{E_\omega\}_{\omega\in\Omega}\subset M_N(\HH)$ such that
\begin{equation}
 E_\omega\ge0\quad(\omega\in\Omega),
 \qquad
 \sum_{\omega\in\Omega}E_\omega=I_N.
 \label{eq:qpovm}
\end{equation}
The probability of outcome $\omega$ in the state $\rho\in\mathcal D_N(\HH)$ is the Born rule
\begin{equation}
 p(\omega\mid\rho):=\Tr_{\HH}(E_\omega\rho).
 \label{eq:qborn}
\end{equation}
Positivity gives $p(\omega\mid\rho)\ge0$, and the POVM normalization gives
$\sum_{\omega\in\Omega}p(\omega\mid\rho)=1$.
These intrinsic state and effect spaces are used twice later: Proposition~\ref{prop:state-embedding}
identifies their exact complex images, and Section~\ref{sec:qec} defines code correction by requiring
recovery on every density matrix supported on a right-$\HH$ submodule.

\subsection{The complex symplectic representation}

To compare the intrinsic theory with complex quantum mechanics, choose imaginary units
$i,j\in\HH$ satisfying $i^2=j^2=-1$ and $ij=-ji$, and set
$k:=ij$.  The selected complex slice is
\begin{equation}
 \CC_i:=\{a+bi:a,b\in\RR\}.
 \label{eq:complex-slice}
\end{equation}
Every quaternion has a unique decomposition $q=z+wj$ with $z,w\in\CC_i$, and
$jz=\overline z\,j$.  Hence every matrix $A\in M_N(\HH)$ has a unique decomposition
$A=B+Cj$ with $B,C\in M_N(\CC_i)$.  Define
\begin{equation}
 \chi(A):=
 \begin{pmatrix}
 B&C\\[-1mm]
 -\overline C&\overline B
 \end{pmatrix}\in M_{2N}(\CC).
 \label{eq:chi}
\end{equation}
For $A,D\in M_N(\HH)$, the map $\chi$ is injective and real-linear, and it satisfies
\begin{equation}
 \chi(AD)=\chi(A)\chi(D),
 \qquad
 \chi(A^\dagger)=\chi(A)^\dagger.
 \label{eq:chi-star}
\end{equation}
Thus $\chi$ faithfully represents the intrinsic right-$\HH$-linear operator algebra by complex
matrices of twice the dimension.

For every $s\in\mathbb N$, let
$\kappa_s:\CC^s\to\CC^s$ denote entrywise complex conjugation in the standard basis.  These
maps are the same basis-conjugation operation on different complex spaces; the subscript records the
dimension.  To keep later changes of code dimension explicit, for every $n\in\mathbb N$ define
\begin{align}
 J_n&:=\begin{pmatrix}0&I_n\\-I_n&0\end{pmatrix}\in M_{2n}(\CC),\notag\\
 \Theta_n&:=J_n\kappa_{2n},
 \qquad
 \alpha_n(X):=J_n\overline XJ_n^{-1}.
 \label{eq:Theta-family}
\end{align}
The map $\Theta_n$ is antiunitary, satisfies $\Theta_n^2=-I_{2n}$, and induces the
conjugate-linear involution $\alpha_n$ on $M_{2n}(\CC)$.  In the physical dimension $N$ we use the
abbreviations
\begin{equation}
 J:=J_N,
 \qquad
 \Theta:=\Theta_N,
 \qquad
 \alpha_\Theta:=\alpha_N.
 \label{eq:Theta}
\end{equation}
Thus
\begin{equation}
 \alpha_\Theta(X)=\Theta X\Theta^{-1}=J\overline XJ^{-1},
 \qquad X\in M_{2N}(\CC).
 \label{eq:alpha-theta}
\end{equation}
For $v,w\in\CC^{2N}$, antiunitarity means
$\Theta(zv)=\overline z\,\Theta(v)$ and
$\langle\Theta v,\Theta w\rangle_{\CC}=\overline{\langle v,w\rangle_{\CC}}$.
The doubled complex space is therefore constrained by a symplectic antiunitary symmetry.  The next
proposition gives the exact constraint and a practical test for deciding whether a complex matrix is
the image of a quaternionic operator.  This standard quaternionic linear-algebra characterization is treated, for
example, by Zhang \cite{Zhang1997}, and its use in quaternionic quantum dynamics is developed by
Graydon \cite{Graydon2013}.  The short proof fixes the signs and the antiunitary
convention used in all later fixed-point calculations.

\begin{proposition}[Characterization of the symplectic image]\label{prop:symplectic-image}
For a matrix $X\in M_{2N}(\CC)$, the following statements are equivalent:
\begin{enumerate}[label=(\roman*),leftmargin=1.6em]
\item there is a unique $A\in M_N(\HH)$ such that $X=\chi(A)$;
\item $X\Theta=\Theta X$;
\item $X=J\overline XJ^{-1}$.
\end{enumerate}
\end{proposition}
\begin{proof}
Because $\kappa_{2N}X=\overline X\kappa_{2N}$, the relation
$X\Theta=\Theta X$ is equivalent to $XJ=J\overline X$, and hence to
$X=J\overline XJ^{-1}$.  Thus (ii) and (iii) are equivalent.  If $X=\chi(B+Cj)$, direct block
multiplication gives $J\overline XJ^{-1}=X$, so (i) implies (iii).  Conversely, write
$X=\begin{psmallmatrix}P&Q\\R&S\end{psmallmatrix}$ with $P,Q,R,S\in M_N(\CC)$.  Condition (iii)
becomes
\[
 \begin{pmatrix}P&Q\\R&S\end{pmatrix}
 =\begin{pmatrix}\overline S&-\overline R\\-\overline Q&\overline P\end{pmatrix},
\]
so $R=-\overline Q$ and $S=\overline P$.  Therefore $X=\chi(P+Qj)$.  Uniqueness follows from the
unique decomposition $A=B+Cj$ over $\CC_i$.
\end{proof}

Three later arguments use this characterization.  Proposition~\ref{prop:state-embedding}
reconstructs a quaternionic density matrix from a fixed complex one;
Theorem~\ref{thm:choi-quaternionic} reconstructs quaternionic Kraus operators from fixed Choi
eigenvectors; and Corollary~\ref{cor:sector-reduction} transports the same fixed-algebra description
to a chosen quaternionic subalgebra of $\OO$.  Conceptually, the quaternionic operator algebra is
the fixed algebra of $X\mapsto J\overline XJ^{-1}$; the ambient algebra $M_{2N}(\CC)$ contains many
matrices with no quaternionic interpretation.

\subsection{From an algebra embedding to an operational simulation}

The algebra embedding $\chi$ does not yet map normalized quaternionic states to normalized complex
states: its complex trace counts the two conjugate diagonal blocks.  This subsection determines the
normalization, proves that positivity is reflected in both directions, and then verifies the Born
rule.  These steps turn the fixed-point algebra into an exact operational representation.

If $A=B+Cj$, then $C$ appears only in the off-diagonal blocks of $\chi(A)$ and contributes nothing
to the complex trace, while the two diagonal blocks are $B$ and $\overline B$.  Hence
\begin{align}
 \Tr_{\CC}\chi(A)
 &=\Tr_{\CC}B+\Tr_{\CC}(\overline B)\notag\\
 &=2\RePart\Tr_{\CC}B
 =2\Tr_{\HH}A.
 \label{eq:trace-doubling}
\end{align}
The factor two is therefore the trace of a conjugate pair of complex blocks, not an additional
physical postulate.  This relation and the associated state simulation are standard in quaternionic
matrix theory and quaternionic quantum dynamics \cite{Zhang1997,Graydon2013}.

\begin{proposition}[Normalized state embedding]\label{prop:state-embedding}
The map
\begin{equation}
 \rho\longmapsto\widetilde\rho:=\tfrac12\chi(\rho)
 \label{eq:stateembed}
\end{equation}
sends $\mathcal D_N(\HH)$ bijectively onto
\begin{equation}
\begin{aligned}
 \mathcal D_{2N}^{\Theta}(\CC):=
 \{\sigma\in M_{2N}(\CC):{}&\ \sigma=\sigma^\dagger,\ \sigma\ge0,\\
 &\Tr_{\CC}\sigma=1,\ \Theta\sigma\Theta^{-1}=\sigma\}.
\end{aligned}
 \label{eq:theta-state-space}
\end{equation}
\end{proposition}
\begin{proof}
Let $\rho\in\mathcal D_N(\HH)$.  The quaternionic spectral theorem for Hermitian matrices
\cite{Zhang1997} applies because $\rho=\rho^\dagger$: there are
$U\in\Sp(N)$ and real eigenvalues $\lambda_1,\ldots,\lambda_N$ such that
$\rho=UDU^\dagger$, where $D=\operatorname{diag}(\lambda_1,\ldots,\lambda_N)$.  Positivity of
$\rho$ is equivalent to $\lambda_a\ge0$ for all $a$, and
\[
 \chi(\rho)=\chi(U)
 \begin{pmatrix}D&0\\0&D\end{pmatrix}
 \chi(U)^\dagger\ge0.
\]
Equations~\eqref{eq:chi-star}, \eqref{eq:trace-doubling}, and
Proposition~\ref{prop:symplectic-image} then show that $\widetilde\rho$ is Hermitian, positive,
trace one, and $\Theta$-fixed.

Conversely, let $\sigma\in\mathcal D_{2N}^{\Theta}(\CC)$.  Proposition~\ref{prop:symplectic-image}
gives a unique $A\in M_N(\HH)$ with $\sigma=\chi(A)$, and Hermiticity of $\sigma$ implies
$A=A^\dagger$.  Write $A=UDU^\dagger$ by the same spectral theorem.  Then
$\sigma=\chi(U)\operatorname{diag}(D,D)\chi(U)^\dagger$.  Since $\sigma\ge0$, every real diagonal
entry of $D$ is nonnegative, so $A\ge0$.  Trace doubling gives $\Tr_{\HH}A=1/2$; hence
$\rho:=2A$ lies in $\mathcal D_N(\HH)$ and $\sigma=\chi(\rho)/2$.  Uniqueness follows from
injectivity of $\chi$.
\end{proof}

State normalization and effect normalization play different roles.  The state representative is
rescaled because its complex trace doubles.  By contrast, for a quaternionic POVM
$\{E_\omega\}_{\omega\in\Omega}$ set $\widetilde E_\omega:=\chi(E_\omega)$.  Then
\[
 \sum_{\omega\in\Omega}\widetilde E_\omega
 =\chi(I_N)=I_{2N},
\]
so the effects already have the correct complex normalization; dividing them by two would instead
produce $I_{2N}/2$ and halve every probability.  Multiplicativity and trace doubling give
\begin{equation}
 \Tr_{\CC}(\widetilde E_\omega\widetilde\rho)
 =\frac12\Tr_{\CC}\chi(E_\omega\rho)
 =\Tr_{\HH}(E_\omega\rho)
 =p(\omega\mid\rho).
 \label{eq:born-simulation}
\end{equation}

The rank also doubles, but for a different reason than the trace normalization.  By the
quaternionic spectral theorem, a Hermitian matrix has a quaternionic orthonormal eigenbasis and real
eigenvalues.  Under $\chi$, each quaternionic eigenvector generates a two-dimensional complex
$\Theta$-pair with the same eigenvalue.  Consequently, if
$\operatorname{rank}_{\HH}$ denotes the quaternionic dimension of the image and
$\operatorname{rank}_{\CC}$ the ordinary complex rank, then
\begin{equation}
 \operatorname{rank}_{\CC}\widetilde\rho
 =2\operatorname{rank}_{\HH}\rho.
 \label{eq:rank-doubling}
\end{equation}
The spectral and rank statements follow from the same quaternionic spectral theorem
\cite{Zhang1997}.  A pure quaternionic
state therefore becomes a complex state of rank two rather than a complex rank-one projector.  In
particular, the unique state $[1]\in\mathcal D_1(\HH)$ maps to $I_2/2$.  The phrase ``doubling the
dimension'' means exact simulation within the $\Theta$-fixed sector, not identification with all
states of $\CC^{2N}$.  A direct eigenspace argument is included in Appendix~\ref{app:symplectic-details}.
\section{Quaternionic channels and Choi liftability}\label{sec:channels}

The goal of this section is twofold.  First, intrinsic quaternionic Kraus maps are represented as
complex CPTP maps while preserving every prepare--evolve--measure probability.  Second, the converse
problem is solved: given a complex CPTP map on the doubled space, determine from the map itself
whether it has a quaternionic Kraus realization.  The Choi matrix provides the representation-
independent answer.  We begin by fixing the channel terminology used in both directions.

Let $d,d'\in\mathbb N$ and let
$\Psi:M_d(\CC)\to M_{d'}(\CC)$ be complex-linear.  For a complex Hermitian matrix $X$, the
notation $X\ge0$ means $v^\dagger Xv\ge0$ for every $v\in\CC^d$.  The map $\Psi$ is
\emph{positive} when $X\ge0$ implies $\Psi(X)\ge0$, and \emph{completely positive} (CP)
when, for every auxiliary complex dimension $r\in\mathbb N$, the amplification
\begin{equation}
 \id_r\otimes\Psi:
 M_r(\CC)\otimes M_d(\CC)\longrightarrow
 M_r(\CC)\otimes M_{d'}(\CC)
 \label{eq:complete-positivity-definition}
\end{equation}
is positive.  It is \emph{trace-preserving} (TP) when
$\Tr_{\CC}\Psi(X)=\Tr_{\CC}X$ for every $X\in M_d(\CC)$, and \emph{unital} when
$\Psi(I_d)=I_{d'}$.  A map that is both completely positive and trace-preserving is abbreviated
\emph{CPTP}.

The Hilbert--Schmidt inner product on $M_d(\CC)$ is
\begin{equation}
 \langle X,Y\rangle_{\rm HS}:=\Tr_{\CC}(X^\dagger Y).
 \label{eq:hilbert-schmidt-inner-product}
\end{equation}
The \emph{Hilbert--Schmidt adjoint} of $\Psi$ is the unique complex-linear map
$\Psi^*:M_{d'}(\CC)\to M_d(\CC)$ satisfying
\begin{equation}
 \begin{aligned}
 \langle Y,\Psi(X)\rangle_{\rm HS}
 &=\langle\Psi^*(Y),X\rangle_{\rm HS},\\[-1mm]
 &\hspace{-8mm}X\in M_d(\CC),\qquad Y\in M_{d'}(\CC).
 \end{aligned}
 \label{eq:hilbert-schmidt-adjoint}
\end{equation}
A map is TP exactly when its Hilbert--Schmidt adjoint is unital.  In the \emph{Schr\"odinger
picture} the channel acts on states, $\rho\mapsto\Psi(\rho)$; in the dual \emph{Heisenberg
picture} the adjoint $\Psi^*$ acts on effects and observables.  For quaternionic maps the analogous
adjoint is defined using the real trace pairing:
\begin{equation}
 \Tr_{\HH}\!\left[E\,\Phi(\rho)\right]
 =\Tr_{\HH}\!\left[\Phi^*(E)\rho\right]
 \qquad(E,\rho\in M_N(\HH)).
 \label{eq:schrodinger-heisenberg-duality}
\end{equation}
These definitions are used throughout Sections~\ref{sec:channels}--\ref{sec:qcomposition}.  We
first describe intrinsic quaternionic Kraus channels and their exact complex simulation, then give
a Kraus-independent Choi test.

\subsection{Quaternionic channels and their exact complex simulation}

Let $m\in\mathbb N$ and let $A_1,\ldots,A_m\in M_N(\HH)$.  In the Schr\"odinger picture, a
quaternionic channel is the real-linear map $\Phi:M_N(\HH)\to M_N(\HH)$ defined by
\begin{equation}
 \Phi(\rho):=\sum_{r=1}^{m}A_r\rho A_r^\dagger,
 \qquad
 \sum_{r=1}^{m}A_r^\dagger A_r=I_N.
 \label{eq:qchannel}
\end{equation}
The matrices $A_r$ are called Kraus operators.  The normalization condition makes $\Phi$
trace-preserving, and the dual Heisenberg map is
$\Phi^{*}(E)=\sum_{r=1}^{m}A_r^{\dagger}EA_r$.  The Kraus form also preserves positivity at
every algebraic matrix level:
\[
 \id_s\otimes\Phi:M_s(M_N(\HH))\cong M_{sN}(\HH)
 \longrightarrow M_{sN}(\HH)
\]
is positive for each $s\in\mathbb N$.  This matrix amplification is an operator-algebraic
criterion and does not, by itself, assert a canonical tensor product of two independent
right-quaternionic systems.  In the finite-dimensional single-system setting,
Eq.~\eqref{eq:qchannel} is the operational definition adopted for an intrinsic quaternionic
channel.  Quaternionic Kraus channels, complex projections of completely positive maps, and exact
complex simulation are established results \cite{Asorey2007,Graydon2013}; the notation is fixed
here so that the later liftability and error-correction statements are unambiguous.

Define $K_r:=\chi(A_r)\in M_{2N}(\CC)$.  The represented map is
\begin{equation}
 \widetilde\Phi:M_{2N}(\CC)\longrightarrow M_{2N}(\CC),
 \qquad
 X\longmapsto\sum_{r=1}^{m}K_rXK_r^\dagger.
 \label{eq:cchannel}
\end{equation}
It is CP by its Kraus form, and
$\sum_rK_r^\dagger K_r=\chi(\sum_rA_r^\dagger A_r)=I_{2N}$ makes it TP.  For
$\rho\in\mathcal D_N(\HH)$, multiplicativity of $\chi$ gives the intertwining identity
\begin{align}
 \widetilde\Phi(\widetilde\rho)
 &=\sum_{r=1}^{m}\chi(A_r)\frac{\chi(\rho)}2\chi(A_r)^\dagger\notag\\
 &=\frac12\chi\!\left(\sum_{r=1}^{m}A_r\rho A_r^\dagger\right)
 =\frac12\chi(\Phi(\rho)).
 \label{eq:intertwine}
\end{align}
The right-hand side is $\Theta$-fixed by Proposition~\ref{prop:symplectic-image}, so the represented
channel preserves the quaternionic state sector.  For a quaternionic effect $E\in M_N(\HH)$,
trace doubling then yields
\begin{align}
 \Tr_{\CC}[\chi(E)\widetilde\Phi(\widetilde\rho)]
 &=\frac12\Tr_{\CC}\chi(E\Phi(\rho))\notag\\
 &=\Tr_{\HH}[E\Phi(\rho)].
 \label{eq:experiment-intertwining}
\end{align}
Thus state preparation, evolution, and measurement are reproduced exactly inside the constrained
complex model.

\subsection{Kraus-index dilation without a bipartite tensor product}

Noncommutativity prevents two right-$\HH$ modules from determining a canonical subsystem tensor
product, but a Kraus label is not itself a second physical subsystem.  A finite Kraus family can be
stored in a direct-sum multiplicity module, which is enough to dilate the Heisenberg adjoint of a
single-system channel.  This distinction is what allows the channel and error-correction theory to be
formulated before any choice of multipartite composition.  The construction is the finite-dimensional
quaternionic instance of Stinespring theory for Hilbert modules
\cite{Paschke1973,BhatRameshSumesh2012,Skeide2012}.

\begin{proposition}[Kraus-index direct-sum dilation]\label{prop:direct-sum-dilation}
Let $\Phi$ be the quaternionic channel in Eq.~\eqref{eq:qchannel}.  On the right quaternionic Hilbert
module
\[
 \mathfrak H_m:=\bigoplus_{r=1}^{m}\HH^N
\]
with componentwise right scalar action and inner product
$\langle\bigoplus_r\xi_r,\bigoplus_r\eta_r\rangle:=\sum_r
\langle\xi_r,\eta_r\rangle_{\HH}$, define
\begin{equation}
 V\psi:=\bigoplus_{r=1}^{m}A_r\psi,
 \qquad
 \pi(E):=\bigoplus_{r=1}^{m}E.
 \label{eq:direct-sum-dilation-data}
\end{equation}
Then $V:\HH^N\to\mathfrak H_m$ is a right-$\HH$-linear isometry,
$\pi:M_N(\HH)\to\End_{\HH}(\mathfrak H_m)$ is a unital $*$-representation, and
\begin{equation}
 \Phi^*(E)=V^\dagger\pi(E)V
 \qquad(E\in M_N(\HH)).
 \label{eq:direct-sum-dilation}
\end{equation}
\end{proposition}
\begin{proof}
For $\xi=\bigoplus_r\xi_r\in\mathfrak H_m$, the adjoint is
$V^\dagger\xi=\sum_rA_r^\dagger\xi_r$.  Therefore
\begin{align}
 V^\dagger V\psi
 &=\sum_rA_r^\dagger A_r\psi=\psi,\notag\\
 V^\dagger\pi(E)V\psi
 &=\sum_rA_r^\dagger EA_r\psi=\Phi^*(E)\psi.
 \label{eq:direct-sum-dilation-check}
\end{align}
Block-diagonal multiplication also gives $\pi(EF)=\pi(E)\pi(F)$,
$\pi(E^\dagger)=\pi(E)^\dagger$, and $\pi(I_N)=I_{\mathfrak H_m}$.
\end{proof}

The proposition solves a dilation problem, not the subsystem-composition problem.  No partial trace,
locality notion, or independent environment state is produced by the direct sum.  Its later use is
more limited and more precise: the recovery channel constructed in
Theorem~\ref{thm:KLH} has such a dilation, so its implementation does not depend on a canonical
quaternionic bipartite tensor product.  Genuine erasure and complementary-channel statements are
deferred to Section~\ref{sec:qcomposition}, where a composite and a discard map are specified first.

\subsection{The Choi matrix and a Kraus-independent symmetry test}

A channel can have many different Kraus families, so the condition that a chosen family commute
with $\Theta$ is not yet an intrinsic test.  The Choi matrix removes this dependence.  The ordinary
complex Choi representation and the equivalence between complete positivity and positivity of the
Choi matrix are due to Choi \cite{Choi1975}.  Set
$d:=2N$, let $\{|\alpha\rangle\}_{\alpha=1}^{d}$ be the standard basis of $\CC^d$, and let
$\Psi:M_d(\CC)\to M_d(\CC)$ be a complex linear map.  Its Choi matrix is
\begin{equation}
 C_\Psi:=\sum_{\alpha,\beta=1}^{d}
 |\alpha\rangle\langle\beta|\otimes
 \Psi(|\alpha\rangle\langle\beta|)
 \in M_{d^2}(\CC).
 \label{eq:choi-definition}
\end{equation}
The first tensor factor is the input factor and the second is the output factor.  The partial trace
over the output factor is the linear map
$\Tr_2:M_d(\CC)\otimes M_d(\CC)\to M_d(\CC)$ determined by
\begin{equation}
 \Tr_2(A\otimes B):=\Tr_{\CC}(B)A.
 \label{eq:partial-trace-definition}
\end{equation}
With this convention, $\Psi$ is trace-preserving exactly when $\Tr_2C_\Psi=I_d$.
Suppose $\ell\in\mathbb N$ and
$\Psi(X)=\sum_{s=1}^{\ell}L_sXL_s^\dagger$ with
$L_s=(L^{(s)}_{\alpha\beta})\in M_d(\CC)$.  Column vectorization is the map
\begin{equation}
 |L_s\rangle\!\rangle:=\operatorname{vec}(L_s)
 :=\sum_{\alpha,\beta=1}^{d}L^{(s)}_{\alpha\beta}
 |\beta\rangle\otimes|\alpha\rangle
 \in\CC^d\otimes\CC^d.
 \label{eq:vectorization-definition}
\end{equation}
With this convention,
\begin{equation}
 C_\Psi=\sum_{s=1}^{\ell}|L_s\rangle\!\rangle\langle\!\langle L_s|.
 \label{eq:choi-kraus}
\end{equation}
Under the standard identification $\CC^d\otimes\CC^d\cong\CC^{d^2}$, the conjugation on
the Choi-vector space is $\kappa_{d^2}$ from the family defined before
Eq.~\eqref{eq:Theta}.  Define
\begin{equation}
 \mathfrak C:=(J\otimes J)\kappa_{d^2}.
 \label{eq:choi-real}
\end{equation}
Because $J^2=-I_d$, the two minus signs cancel and $\mathfrak C^2=I_{d^2}$.  Choi's construction \cite{Choi1975} packages the action of the channel on all matrix units into one
positive operator.  Quaternionic completely positive maps, complex projections, exact complex
simulations, and intrinsic Choi--Kraus-type representations are already available
\cite{Asorey2007,Graydon2013,BhardwajSharmaKaushik2021}.  For the real structure of imaginarity
theory, reality of the Choi matrix is tied to real Kraus representations, and later
channel-resource formulations use the Choi state directly \cite{HickeyGour2018,ChenLei2025}.  The
criterion below addresses the corresponding $\Theta^2=-I$ membership problem on the doubled
complex space.  The point is to identify the fixed Choi vectors with $\chi(M_N(\HH))$ and to
descend complex trace preservation to a quaternionic Kraus family.

\begin{theorem}[Antiunitary Choi fixed-point criterion]\label{thm:choi-quaternionic}
Let $\Psi:M_d(\CC)\to M_d(\CC)$ be CPTP, with $d=2N$.  The following conditions are
equivalent:
\begin{enumerate}[label=(\roman*),leftmargin=1.6em]
\item there are $\ell\in\mathbb N$ and matrices
$L_1,\ldots,L_\ell\in M_d(\CC)$ forming a Kraus family for $\Psi$ such that
$L_s\Theta=\Theta L_s$ for every $s\in\{1,\ldots,\ell\}$;
\item its Choi matrix satisfies
\begin{equation}
 \mathfrak C C_\Psi\mathfrak C^{-1}=C_\Psi.
 \label{eq:choi-fixed}
\end{equation}
\end{enumerate}
When these conditions hold, a compatible Kraus family can be chosen with every $L_s$ in the
image of $\chi$, and $\Psi$ is the complex representation of a quaternionic channel.
\end{theorem}
\begin{proof}
The relation $L\Theta=\Theta L$ is equivalent to
$L=J\overline LJ^{-1}$.  For the vectorization convention
in Eq.~\eqref{eq:vectorization-definition}, the standard identity
$\operatorname{vec}(AXB)=(B^{\mathsf T}\otimes A)\operatorname{vec}(X)$ gives
\begin{align}
 \operatorname{vec}(J\overline LJ^{-1})
 &=\bigl((J^{-1})^{\mathsf T}\otimes J\bigr)
   \kappa_{d^2}\operatorname{vec}(L)\notag\\
 &=(J\otimes J)\kappa_{d^2}\operatorname{vec}(L)
 =\mathfrak C|L\rangle\!\rangle,
 \label{eq:choi-vector-fixed-calculation}
\end{align}
because $J^{-1}=-J$ and $J^{\mathsf T}=-J$.  Thus
\[
 L\Theta=\Theta L
 \quad\Longleftrightarrow\quad
 \mathfrak C|L\rangle\!\rangle=|L\rangle\!\rangle.
\]
If every Kraus operator has this property, each rank-one term in
Eq.~\eqref{eq:choi-kraus} is fixed by conjugation with $\mathfrak C$, proving (i)$\Rightarrow$(ii).

Conversely, assume $C_\Psi\ge0$ and
$\mathfrak C C_\Psi\mathfrak C^{-1}=C_\Psi$.  Since $\mathfrak C$ is an antiunitary involution,
each eigenspace $\mathcal E_\lambda$ of $C_\Psi$ is $\mathfrak C$-stable.  Its fixed set
\[
 \mathcal E_\lambda^{\mathfrak C}
 :=\{v\in\mathcal E_\lambda:\mathfrak C v=v\}
\]
is a real Hilbert space and
$\mathcal E_\lambda=\mathcal E_\lambda^{\mathfrak C}
 \oplus i\mathcal E_\lambda^{\mathfrak C}$.  The sum is direct: if a vector is both fixed and
of the form $iw$ with $w$ fixed, anti-linearity gives $iw=\mathfrak C(iw)=-iw$, hence the vector
vanishes.
Hence one may choose an orthonormal eigenbasis
$\{v_{\lambda,t}\}_t$ with every $v_{\lambda,t}$ fixed by $\mathfrak C$.
For each nonzero eigenvalue, define
\[
 |L_{\lambda,t}\rangle\!\rangle
 :=\sqrt{\lambda}\,v_{\lambda,t}.
\]
Then
$C_\Psi=\sum_{\lambda,t}
 |L_{\lambda,t}\rangle\!\rangle\langle\!\langle L_{\lambda,t}|$,
so the reshaped matrices $L_{\lambda,t}$ form a Kraus family for $\Psi$.
Equation~\eqref{eq:choi-vector-fixed-calculation} shows that each of them commutes with
$\Theta$.  Proposition~\ref{prop:symplectic-image} therefore gives unique
$A_{\lambda,t}\in M_N(\HH)$ such that
$L_{\lambda,t}=\chi(A_{\lambda,t})$.  Finally, trace preservation is encoded by
$\Tr_2 C_\Psi=I_d$, with $\Tr_2$ defined in Eq.~\eqref{eq:partial-trace-definition}.  With the input--output vectorization convention
of Eq.~\eqref{eq:vectorization-definition},
\[
 \Tr_2\!\left(|L\rangle\!\rangle\langle\!\langle L|\right)
   =(L^\dagger L)^{\mathsf T}.
\]
Consequently $\Tr_2 C_\Psi=I_d$ is equivalent to
$(\sum_{\lambda,t}L_{\lambda,t}^\dagger L_{\lambda,t})^{\mathsf T}=I_d$, and hence to
$\sum_{\lambda,t}L_{\lambda,t}^\dagger L_{\lambda,t}=I_d$.  Multiplicativity and
injectivity of $\chi$ then imply
$\sum_{\lambda,t}A_{\lambda,t}^\dagger A_{\lambda,t}=I_N$.  Hence this Kraus family is
exactly the complex representation of a trace-preserving quaternionic channel, proving
(ii)$\Rightarrow$(i).
\end{proof}

\begin{corollary}[Antiunitary equivariance]\label{cor:choi-equivariance}
Recall the matrix involution $\alpha_\Theta$ from Eq.~\eqref{eq:alpha-theta}.
For a complex CPTP map $\Psi:M_d(\CC)\to M_d(\CC)$, the two conditions in
Theorem~\ref{thm:choi-quaternionic} are also equivalent to
\begin{equation}
 \Psi\!\circ\!\alpha_\Theta
 =\alpha_\Theta\!\circ\!\Psi.
 \label{eq:theta-equivariance}
\end{equation}
\end{corollary}
\begin{proof}
Although $\alpha_\Theta$ is conjugate-linear, the composition
$\alpha_\Theta\circ\Psi\circ\alpha_\Theta$ is complex linear.  If
$\Psi(X)=\sum_sL_sXL_s^\dagger$, then
$\alpha_\Theta\circ\Psi\circ\alpha_\Theta$ has Kraus operators
$\alpha_\Theta(L_s)=J\overline{L_s}J^{-1}$.  By
Eq.~\eqref{eq:choi-vector-fixed-calculation},
$\operatorname{vec}(\alpha_\Theta(L_s))=\mathfrak C\operatorname{vec}(L_s)$; using
Eq.~\eqref{eq:choi-kraus} therefore gives
\begin{equation}
 C_{\alpha_\Theta\circ\Psi\circ\alpha_\Theta}
 =\mathfrak C C_\Psi\mathfrak C^{-1}.
 \label{eq:choi-covariance-transform}
\end{equation}
The Choi representation is injective and $\alpha_\Theta^2=\id$.  Hence
Eq.~\eqref{eq:choi-fixed} is equivalent to
$\Psi=\alpha_\Theta\circ\Psi\circ\alpha_\Theta$, which is precisely
Eq.~\eqref{eq:theta-equivariance}.
\end{proof}

The characterization is operationally useful because the Choi matrix is determined by the channel,
whereas individual Kraus families are not.  Liftability can therefore be tested by one linear
fixed-point condition on $C_\Psi$, and a compatible quaternionic Kraus family can be reconstructed
from a fixed eigenbasis.  It also locates the boundary of exact simulation: a generic complex CPTP
map on $\CC^{2N}$ need not preserve the quaternionic sector.  The forward direction---that quaternionic channels and POVMs are simulated by complex CPTP maps
and complex POVMs---is established \cite{Asorey2007,Graydon2013}.  The complementary statement in
Theorem~\ref{thm:choi-quaternionic} is a Kraus-independent membership test: a single linear
fixed-point equation on $C_\Psi$ decides membership in the symplectic image, a compatible
right-$\HH$-linear Kraus family is reconstructed from a $\mathfrak C$-fixed eigenbasis, and
complex trace preservation descends to the quaternionic normalization
$\sum_rA_r^\dagger A_r=I_N$.  The equivariance formulation is used
in Lemma~\ref{lem:theta-sector-preservation} to control both a corrected channel and its Heisenberg
adjoint.  Example~\ref{ex:channel-embed} then shows that this fixed-point test still includes
channels with genuinely noncommuting quaternionic phases.

\begin{example}[Noncommuting quaternionic phases]\label{ex:channel-embed}
Let $U:=\operatorname{diag}(i,j)\in\Sp(2)$ and
$|+\rangle:=(1,1)^{\mathsf T}/\sqrt2\in\HH^2$.  Then
\begin{equation}
 U|+\rangle\langle+|U^\dagger
 =\frac12\begin{pmatrix}1&-k\\ k&1\end{pmatrix}.
 \label{eq:qphase-example}
\end{equation}
The off-diagonal entries record the noncommutativity of the quaternionic phases.  The embedded
input $\chi(|+\rangle\langle+|)/2$ has complex rank two, and
Eq.~\eqref{eq:intertwine} reproduces the evolved quaternionic state exactly.  The example is
important because liftability is not a disguised ``real-matrix'' condition: the channel contains
genuinely noncommuting quaternionic phases, yet its doubled complex Choi matrix satisfies the
symplectic fixed-point criterion.
\end{example}

\section{Exact quaternionic quantum error correction}\label{sec:qec}

The question is whether a noisy quaternionic evolution can be reversed on a chosen logical
subspace.  Exact correction of complex subspaces is governed by the Knill--Laflamme theorem, and
operator quantum error correction extends the same compression principle to logical algebras
\cite{KnillLaflamme1997,KribsLaflammePoulinLesosky2006}.  A recent quaternionic treatment develops
a concrete code construction: quaternionic analogues of Pauli operators, an encoding of logical
qubits, and a quaternionic extension of the five-qubit code with an explicit syndrome list
\cite{Nyirahafashimana2025}.  That contribution is centered on a concrete code construction and its syndrome analysis;
the general right-module criterion, the restriction on admissible compression scalars, and the
intrinsic recovery developed below are not among the results stated there.  Kubischta and Teixeira
instead use physical time reversal on ordinary complex spin systems to show that symmetry makes
particular parity sectors of the Knill--Laflamme conditions automatic \cite{KubischtaTeixeira2026}.
That selection rule concerns Kramers-invariant complex codes and specified error algebras; it does
not supply the right-$\HH$-linear equivalence or recovery considered here.  Numerical performance
comparisons also require an explicit resource convention, because quaternionic dimension $N$ is
represented on complex dimension $2N$ and ranks double under $\chi$; no threshold comparison is
used below.  The issue addressed here is narrower: how the standard finite-dimensional criterion
descends to a right-$\HH$-linear subspace code, which scalars can occur in the compression, and how
the recovery can be written intrinsically.  The
setup is single-system and does not assume a tensor product of independent quaternionic systems.

Let $\mathcal C\subset\HH^N$ be a right-$\HH$ submodule of quaternionic dimension
$K\in\mathbb N$.  Let $P\in M_N(\HH)$ be its orthogonal projector in the sense of
Eq.~\eqref{eq:orthogonal-projection-definition}; explicitly,
$P=P^\dagger=P^2$ and $\operatorname{ran}P=\mathcal C$.  A state
$\rho\in\mathcal D_N(\HH)$ is supported on the code when $P\rho P=\rho$.

Let $m\in\mathbb N$ and let $E_1,\ldots,E_m\in M_N(\HH)$ satisfy
$\sum_{a=1}^{m}E_a^\dagger E_a=I_N$.  The noise channel is
\begin{equation}
 \mathcal N(\rho):=\sum_{a=1}^{m}E_a\rho E_a^\dagger.
 \label{eq:noise-channel}
\end{equation}
The code is exactly correctable when there exists a quaternionic channel $\mathcal R$ such that
$\mathcal R\!\circ\!\mathcal N(\rho)=\rho$ for every density matrix supported on
$\mathcal C$.

\subsection{Why the Knill--Laflamme coefficients must be real}

In complex quantum mechanics, exact correction is characterized by compressions of
$E_a^\dagger E_b$ that act as scalars on the code.  Over $\HH$, the word ``scalar'' needs care.
The logical operator algebra is $M_K(\HH)$, whose center is $\RR I_K$, not all quaternionic
multiples of the identity.  A non-real quaternion multiplying code vectors from the left is a
logical operation rather than an invisible scalar.  Two inequivalent formalizations of ``acts as a
scalar'' must therefore be kept apart.

The first places the scalar to the right of the inner product, which is the position dictated by
right-linearity in the second argument.  Suppose a right-$\HH$-linear map
$T:\mathcal C\to\mathcal C$ and a quaternion $q\in\HH$ satisfy
\begin{equation}
 \langle\psi,T\phi\rangle_{\HH}
 =\langle\psi,\phi\rangle_{\HH}q
 \quad\text{for every }\psi,\phi\in\mathcal C.
 \label{eq:sesqui}
\end{equation}
Replace $\phi$ by $\phi r$, where $r\in\HH$.  Right-linearity of $T$ and of the second
argument of the inner product gives
\[
 \langle\psi,\phi\rangle_{\HH}(qr)
 =\langle\psi,\phi\rangle_{\HH}(rq).
\]
Taking a unit vector $\phi$ and setting $\psi=\phi$ yields $qr=rq$ for every
$r\in\HH$.  Hence $q$ belongs to the center $Z(\HH)=\RR$.  With $q$ real,
Eq.~\eqref{eq:sesqui} becomes
$\langle\psi,(T-qI_{\mathcal C})\phi\rangle_{\HH}=0$ for every $\psi,\phi\in\mathcal C$;
nondegeneracy of the quaternionic inner product then gives $T=qI_{\mathcal C}$.

The second formalization places the scalar to the left of the projector, $PFP=cP$ with
$c\in\HH$.  This is a well-posed right-$\HH$-linear equation, and it does \emph{not} by itself
force $c$ to be real; Example~\ref{ex:nonreal-compression} exhibits a normalized error family whose
left compressions have a genuinely imaginary coefficient.  Reality of the Gram matrix in the
theorem below is therefore not a formal consequence of Eq.~\eqref{eq:sesqui}: it is proved inside
the theorem by the antiunitary step, which uses exact correctability as a hypothesis.  What
Eq.~\eqref{eq:sesqui} records is the structural reason behind that outcome.  A correctable cross
product must act on the code as a scalar in the operational sense, in which the coefficient is read
off from the inner product with its intrinsic order, and only real coefficients survive that
requirement.

\begin{example}[A non-real left compression that is not correctable]\label{ex:nonreal-compression}
Let $N=K=2$, $\mathcal C=\HH^2$, and $P=I_2$.  Put
\[
 E_1:=\tfrac1{\sqrt2}I_2,
 \qquad
 E_2:=\tfrac1{\sqrt2}\,iI_2 .
\]
Then $E_1^\dagger E_1+E_2^\dagger E_2
=\tfrac12I_2+\tfrac12\overline i\,iI_2=I_2$, so $\{E_1,E_2\}$ is an admissible quaternionic noise
channel, while
\[
 PE_1^\dagger E_2P=\tfrac{i}{2}\,P .
\]
The left compression coefficient is not real.  The family is nevertheless not correctable.  For
$\psi_\pm:=(1,\pm j)^{\mathsf T}/\sqrt2$ one computes
\[
 \mathcal N(\psi_+\psi_+^\dagger)
 =\mathcal N(\psi_-\psi_-^\dagger)
 =\tfrac12I_2 ,
\]
so two distinct pure code states have the same output and no recovery exists.  In agreement with
Theorem~\ref{thm:KLH}, no real Gram matrix represents these compressions.  The example also shows
that the left-scalar equation is strictly weaker than Eq.~\eqref{eq:sesqui}.
\end{example}

\subsection{The exact criterion}

The complex subspace criterion and its operator-algebraic extensions supply the algebraic reference
point \cite{KnillLaflamme1997,KribsLaflammePoulinLesosky2006}.  Passing to $\HH$ is not a
formal substitution of scalars: the compression coefficients must lie in the real center,
quaternionic dimension one has to be treated separately at the level of states, and sufficiency must
be realized by a right-$\HH$-linear recovery.  The two lemmas isolate the fixed-sector and
algebra-generation steps needed to justify that passage.

The necessity proof uses two finite-dimensional facts.  The first is the channel-and-adjoint
version of Corollary~\ref{cor:choi-equivariance}: it records the antiunitary symmetry preserved by
every represented quaternionic channel and its adjoint.  The second shows that the
quaternionic Hermitian observables generate the full complex operator algebra on the doubled code.
The separate case $K=1$ is treated after the theorem.

\begin{lemma}[Preservation of the quaternionic fixed sector]\label{lem:theta-sector-preservation}
Let $d\in\mathbb N$, let $\Theta_d:\CC^d\to\CC^d$ be antiunitary, and define
$\alpha_{\Theta_d}(X):=\Theta_dX\Theta_d^{-1}$ on $M_d(\CC)$.  Let
$\Psi:M_d(\CC)\to M_d(\CC)$ have a Kraus family
$\Psi(X)=\sum_sL_sXL_s^\dagger$ satisfying $L_s\Theta_d=\Theta_d L_s$ for every $s$.  Then
\begin{equation}
 \Psi\circ\alpha_{\Theta_d}=\alpha_{\Theta_d}\circ\Psi,
 \qquad
 \Psi^*\circ\alpha_{\Theta_d}=\alpha_{\Theta_d}\circ\Psi^*.
 \label{eq:channel-and-adjoint-equivariance}
\end{equation}
Consequently, both $\Psi$ and its Hilbert--Schmidt adjoint $\Psi^*$ preserve the
$\alpha_{\Theta_d}$-fixed sector.  If $Q=\alpha_{\Theta_d}(Q)$ is a projection, then the compressed map
$X\mapsto Q\Psi^*(X)Q$ preserves the fixed part of the corner $QM_d(\CC)Q$.
\end{lemma}
\begin{proof}
The commutation relation is equivalent to $\alpha_{\Theta_d}(L_s)=L_s$, and it also implies
$\alpha_{\Theta_d}(L_s^\dagger)=L_s^\dagger$.  Applying $\alpha_{\Theta_d}$ to each Kraus term proves
both identities in Eq.~\eqref{eq:channel-and-adjoint-equivariance}.  Multiplication by the fixed
projection $Q$ then preserves the same sector.
\end{proof}

A \emph{complex unital $*$-subalgebra} of $M_{2K}(\CC)$ is a complex-linear subspace that
contains $I_{2K}$ and is closed under matrix multiplication and adjoints.  No infinite-dimensional operator-algebra machinery is needed in the next lemma.

\begin{lemma}[Generation by quaternionic Hermitian observables]\label{lem:qherm-generation}
For $K\ge2$, the smallest complex unital $*$-subalgebra of $M_{2K}(\CC)$ containing
$\chi(\operatorname{Herm}_K(\HH))$ is $M_{2K}(\CC)$, where
$\operatorname{Herm}_K(\HH):=\{H\in M_K(\HH):H=H^\dagger\}$.
\end{lemma}
\begin{proof}
Let $e_{ab}$ be the $K\times K$ matrix unit.  The diagonal projections $e_{aa}$ are Hermitian.
For $a\ne b$ and $q\in\{1,i,j,k\}$, the matrix
\[
 h_{ab}(q):=q e_{ab}+\overline q\,e_{ba}
\]
is Hermitian, and $h_{ab}(q)e_{bb}=q e_{ab}$.  Products with $e_{ba}$ then give
$q e_{aa}$.  Hence the real algebra generated by the Hermitian quaternionic matrices contains
every quaternionic matrix unit and is all of $M_K(\HH)$.

For the complexification step, take $X\in M_{2K}(\CC)$ and use the logical-dimension involution
$\alpha_K(X)=J_K\overline XJ_K^{-1}$.  Define
\[
 X_+:=\frac{X+\alpha_K(X)}2,
 \qquad
 X_-:=\frac{X-\alpha_K(X)}{2i}.
\]
Conjugate-linearity gives
$\alpha_K(X_+)=X_+$ and $\alpha_K(X_-)=X_-$.  Hence both matrices belong to
$\chi(M_K(\HH))$ by
Proposition~\ref{prop:symplectic-image}, and $X=X_++iX_-$.  Thus the complex linear span of
$\chi(M_K(\HH))$ is $M_{2K}(\CC)$, so the generated complex unital $*$-subalgebra is the full matrix algebra.
\end{proof}

For the doubled code, set
\begin{equation}
 Q:=\chi(P),
 \qquad
 \mathcal A_Q:=QM_{2N}(\CC)Q.
 \label{eq:code-corner-algebra}
\end{equation}
The algebra $\mathcal A_Q$ is called the \emph{corner algebra} supported on
$\operatorname{ran}Q$; its multiplicative identity is $Q$, not $I_{2N}$.  Its Hermitian
$\alpha_\Theta$-fixed part is
\begin{equation}
 \mathcal A_{Q,\Theta}^{\rm sa}
 :=\{X\in\mathcal A_Q:X=X^\dagger,\ \alpha_\Theta(X)=X\}.
 \label{eq:fixed-hermitian-corner}
\end{equation}
By Proposition~\ref{prop:symplectic-image}, this real vector space is the symplectic image of the
quaternionic Hermitian operators on $\mathcal C$.  To make that identification explicit, choose an
orthonormal basis of $\mathcal C$ and extend it to an orthonormal basis of $\HH^N$.  The resulting
basis-change matrix $U_{\mathcal C}\in\Sp(N)$ satisfies
$P=U_{\mathcal C}P_0U_{\mathcal C}^{\dagger}$, where
$P_0:=\operatorname{diag}(I_K,0_{N-K})$, and set
$\widehat U_{\mathcal C}:=\chi(U_{\mathcal C})$.  Define the isometry
\begin{align}
 S_K&:\CC^{2K}\longrightarrow\CC^{2N},\notag\\
 S_K(u\oplus v)&:=(u\oplus0_{N-K})\oplus(v\oplus0_{N-K}).
 \label{eq:standard-code-isometry}
\end{align}
Then $S_K^\dagger S_K=I_{2K}$ and $S_KS_K^\dagger=\chi(P_0)$.  Hence
\begin{equation}
 \beta_{\mathcal C}:M_{2K}(\CC)\longrightarrow\mathcal A_Q,
 \qquad
 \beta_{\mathcal C}(X)
 :=\widehat U_{\mathcal C}S_KXS_K^\dagger\widehat U_{\mathcal C}^{\dagger}
 \label{eq:code-corner-identification}
\end{equation}
is a unital $*$-isomorphism from $M_{2K}(\CC)$ onto the corner, with
$\beta_{\mathcal C}(I_{2K})=Q$.  The standard inclusion also intertwines the two symplectic
structures, $S_K\Theta_K=\Theta S_K$, while
$\widehat U_{\mathcal C}\Theta=\Theta\widehat U_{\mathcal C}$.  Consequently,
\begin{equation}
 \beta_{\mathcal C}(\alpha_K(X))
 =\alpha_\Theta(\beta_{\mathcal C}(X)),
 \qquad X\in M_{2K}(\CC).
 \label{eq:code-corner-involution-intertwining}
\end{equation}
Thus the $\alpha_K$-fixed logical algebra is mapped exactly onto the
$\alpha_\Theta$-fixed part of the physical corner.  For
$H\in\operatorname{Herm}_K(\HH)$ write
$\chi_{\mathcal C}(H):=\beta_{\mathcal C}(\chi(H))$.  If $\rho$ is supported on $\mathcal C$,
its logical coordinate $\rho_{\mathcal C}\in\mathcal D_K(\HH)$ satisfies
$\chi(\rho)=\beta_{\mathcal C}(\chi(\rho_{\mathcal C}))$ and therefore
\begin{equation}
 \Tr_{\CC}\!\left[\chi_{\mathcal C}(H)\frac{\chi(\rho)}2\right]
 =\Tr_{\HH}(H\rho_{\mathcal C}).
 \label{eq:code-corner-trace-pairing}
\end{equation}

\begin{theorem}[Quaternionic Knill--Laflamme theorem]\label{thm:KLH}
Assume that the code $\mathcal C\subset\HH^N$ has quaternionic dimension $K\ge2$.  For the
error family $\{E_a\}_{a=1}^{m}$, the following statements are equivalent:
\begin{enumerate}[label=(\roman*),leftmargin=1.6em]
\item there exists a quaternionic channel $\mathcal R$ that exactly corrects
$\mathcal N$ on every state supported on $\mathcal C$;
\item there exists a real positive-semidefinite matrix
$r=(r_{ab})\in M_m(\RR)$ such that
\begin{equation}
 PE_a^\dagger E_bP=r_{ab}P
 \qquad (1\le a,b\le m);
 \label{eq:KLH}
\end{equation}
\item there exists a positive-semidefinite complex matrix
$\lambda=(\lambda_{ab})\in M_m(\CC)$ such that the doubled code and errors obey
\begin{equation}
 Q\,\chi(E_a)^\dagger\chi(E_b)\,Q=\lambda_{ab}Q
 \qquad(1\le a,b\le m).
 \label{eq:complex-KL-on-corner}
\end{equation}
\end{enumerate}
When these conditions hold, the matrix $r$ is uniquely determined by the error family and code.  It
is real symmetric, $\Tr r=1$, and $\operatorname{rank}_{\RR}r$ is the number of nonzero syndrome
directions produced by the orthogonalization in the proof.
\end{theorem}
\begin{proof}
We first prove (i)$\Rightarrow$(iii).  Let
$\mathcal T:=\mathcal R\circ\mathcal N$ and let $\widetilde{\mathcal T}$ be its complex
representation.  Every Kraus operator of $\widetilde{\mathcal T}$ commutes with $\Theta$, so
Lemma~\ref{lem:theta-sector-preservation} applies with $d=2N$ and $\Theta_d=\Theta$.  The \emph{compressed Heisenberg map} is the
Hilbert--Schmidt adjoint of the corrected channel, restricted and compressed to the code corner:
\begin{equation}
 \Gamma:\mathcal A_Q\longrightarrow\mathcal A_Q,
 \qquad
 \Gamma(X):=Q\,\widetilde{\mathcal T}^{\,*}(X)\,Q.
 \label{eq:compressed-heisenberg-map}
\end{equation}
It is completely positive and preserves the fixed Hermitian corner
$\mathcal A_{Q,\Theta}^{\rm sa}$.  Because the identity element of $\mathcal A_Q$ is $Q$, saying
that $\Gamma$ is \emph{unital on the corner} means precisely
$\Gamma(Q)=Q$.

To prove this equality, use
$0\le\widetilde{\mathcal T}^{\,*}(Q)\le I_{2N}$ and define
\[
 B:=Q-\Gamma(Q)=Q[I_{2N}-\widetilde{\mathcal T}^{\,*}(Q)]Q.
\]
The operator $B$ belongs to $\mathcal A_{Q,\Theta}^{\rm sa}$ and is positive.  For every
quaternionic code state $\rho=P\rho P$, exact correction, Proposition~\ref{prop:state-embedding},
and Schr\"odinger--Heisenberg duality give
\begin{align*}
 \Tr_{\CC}\!\left[B\frac{\chi(\rho)}2\right]
 &=1-\Tr_{\CC}\!\left[Q\,
   \widetilde{\mathcal T}\!\left(\frac{\chi(\rho)}2\right)\right]\\
 &=1-\Tr_{\CC}\!\left[Q\frac{\chi(\rho)}2\right]=0.
\end{align*}
By Eq.~\eqref{eq:fixed-hermitian-corner}, $B$ is the image of a Hermitian quaternionic operator on
the code.  Quaternionic density matrices span that real vector space.  Explicitly, for
$H=H^\dagger$ choose $c>\max_j|\lambda_j|$, put
$t:=\Tr_{\HH}(H+cI_K)$, $\rho_H:=(H+cI_K)/t$, and $\rho_*:=I_K/K$; then
$H=t\rho_H-cK\rho_*$.  The real trace pairing is nondegenerate: if $H\ne0$, the
quaternionic spectral theorem gives $\Tr_{\HH}(H^2)=\sum_j\lambda_j^2>0$.  Hence $B=0$
and $\Gamma(Q)=Q$.

Now let $H\in\operatorname{Herm}_K(\HH)$ and use the code-corner representative
$\chi_{\mathcal C}(H)$ from Eq.~\eqref{eq:code-corner-identification}.  The difference
\[
 D_H:=\Gamma(\chi_{\mathcal C}(H))-\chi_{\mathcal C}(H)
\]
also lies in $\mathcal A_{Q,\Theta}^{\rm sa}$.  Exact correction gives, for every quaternionic
code state $\rho$,
\begin{align*}
 \Tr_{\CC}\!\left[D_H\frac{\chi(\rho)}2\right]
 &=\Tr_{\CC}\!\left[\chi_{\mathcal C}(H)\,
 \widetilde{\mathcal T}\!\left(\frac{\chi(\rho)}2\right)\right]\\
 &\quad-\Tr_{\CC}\!\left[\chi_{\mathcal C}(H)\frac{\chi(\rho)}2\right]
 =0.
\end{align*}
The same nondegeneracy argument yields $D_H=0$.  Since $H^2$ is Hermitian,
\[
 \Gamma(\chi_{\mathcal C}(H)^2)=\chi_{\mathcal C}(H)^2=\Gamma(\chi_{\mathcal C}(H))^2.
\]
For a unital completely positive map $\Gamma$, its \emph{multiplicative domain} is the set of
operators $Z$ satisfying both
$\Gamma(Z^\dagger Z)=\Gamma(Z)^\dagger\Gamma(Z)$ and
$\Gamma(ZZ^\dagger)=\Gamma(Z)\Gamma(Z)^\dagger$; on this set $\Gamma$ preserves
left and right multiplication.  The multiplicative-domain theorem therefore places every
$\chi_{\mathcal C}(H)$ in that domain \cite{Choi1974Schwarz}.  Under the $*$-isomorphism $\beta_{\mathcal C}$, Lemma~\ref{lem:qherm-generation}
shows that these operators generate $\mathcal A_Q$, and therefore
\begin{equation}
 \Gamma(X)=X\qquad(X\in\mathcal A_Q).
 \label{eq:gamma-identity-corner}
\end{equation}

It remains to pass from the fixed quaternionic states to every complex state of the doubled code.
Let $\sigma\ge0$ satisfy $\Tr_{\CC}\sigma=1$ and $Q\sigma Q=\sigma$.  From
$\Gamma(Q)=Q$,
\begin{align}
 \Tr_{\CC}[Q\widetilde{\mathcal T}(\sigma)]
 &=\Tr_{\CC}[\widetilde{\mathcal T}^{\,*}(Q)\sigma]\notag\\
 &=\Tr_{\CC}[Q\widetilde{\mathcal T}^{\,*}(Q)Q\sigma]\notag\\
 &=1.
 \label{eq:output-support-calculation}
\end{align}
Put $Y:=\widetilde{\mathcal T}(\sigma)$.  Positivity and trace preservation imply
$\Tr_{\CC}[(I_{2N}-Q)Y]=0$.  The positive operator
$Y^{1/2}(I_{2N}-Q)Y^{1/2}$ therefore vanishes, so $Y=QYQ$.  For every
$X\in\mathcal A_Q$, Eqs.~\eqref{eq:compressed-heisenberg-map} and
\eqref{eq:gamma-identity-corner} then give
\[
 \Tr_{\CC}[X\widetilde{\mathcal T}(\sigma)]
 =\Tr_{\CC}[\Gamma(X)\sigma]
 =\Tr_{\CC}[X\sigma].
\]
The complex trace pairing on $\mathcal A_Q$ is nondegenerate, hence
$\widetilde{\mathcal T}(\sigma)=\sigma$.  The ordinary complex Knill--Laflamme theorem now yields
Eq.~\eqref{eq:complex-KL-on-corner}, proving (iii).

For (iii)$\Rightarrow$(ii), the left-hand side of
Eq.~\eqref{eq:complex-KL-on-corner} is $\chi(PE_a^\dagger E_bP)$ and is therefore fixed by
$\alpha_\Theta$.  Applying this conjugate-linear involution to the right-hand side replaces
$\lambda_{ab}$ by $\overline{\lambda_{ab}}$ while fixing $Q$.  Since $Q\ne0$,
$\lambda_{ab}=\overline{\lambda_{ab}}$.  Thus $r_{ab}:=\lambda_{ab}\in\RR$, and injectivity of
$\chi$ gives Eq.~\eqref{eq:KLH}.  The matrix $r$ remains positive semidefinite.

Finally assume (ii).  Choose a real orthogonal matrix $O\in M_m(\RR)$ satisfying
$OO^{\mathsf T}=I_m$ and
\begin{equation}
 OrO^{\mathsf T}=\operatorname{diag}(d_1,\ldots,d_m),
 \qquad d_\mu\ge0.
 \label{eq:gram-diagonalization}
\end{equation}
Define
$F_\mu:=\sum_{a=1}^{m}O_{\mu a}E_a$.  For each index with $d_\mu>0$, set
\begin{equation}
 V_\mu:=d_\mu^{-1/2}F_\mu P,
 \qquad
 Q_\mu:=V_\mu V_\mu^\dagger.
 \label{eq:partialiso}
\end{equation}
Equation~\eqref{eq:KLH} gives
$V_\mu^\dagger V_\nu=\delta_{\mu\nu}P$, so the ranges of the projections $Q_\mu$ are mutually
orthogonal syndrome sectors.  Define
\begin{equation}
 \mathcal R_0(X):=
 \sum_{\mu:d_\mu>0}V_\mu^\dagger Q_\mu XQ_\mu V_\mu,
 \qquad X\in M_N(\HH).
 \label{eq:recovery}
\end{equation}
For a code state $\rho=P\rho P$,
\begin{align}
 \mathcal N(\rho)
 &=\sum_{\mu:d_\mu>0}d_\mu V_\mu\rho V_\mu^\dagger,\notag\\
 \mathcal R_0(\mathcal N(\rho))
 &=\left(\sum_\mu d_\mu\right)\rho=\rho.
 \label{eq:recovery-on-code}
\end{align}
The last equality follows because compressing
$\sum_aE_a^\dagger E_a=I_N$ with $P$ gives
$\sum_\mu d_\mu=\Tr r=1$.

To complete $\mathcal R_0$ to a trace-preserving channel, put
$Q_{\rm syn}:=\sum_{\mu:d_\mu>0}Q_\mu$, choose a unit vector
$\psi_0\in\mathcal C$, and choose an orthonormal basis
$\{f_t\}_{t=1}^{N-\operatorname{rank}_{\HH}Q_{\rm syn}}$ of
$(I_N-Q_{\rm syn})\HH^N$.  Let $W_t:=\psi_0f_t^\dagger$ and define
\begin{equation}
 \mathcal R(X):=\mathcal R_0(X)+\sum_tW_tXW_t^\dagger.
 \label{eq:recovery-completion}
\end{equation}
The Kraus operators $V_\mu^\dagger Q_\mu$ and $W_t$ satisfy
\[
 \sum_{\mu:d_\mu>0}Q_\mu+\sum_t f_tf_t^\dagger=I_N,
\]
so $\mathcal R$ is trace preserving.  The added term acts only on the orthogonal complement of the
syndrome space and prepares the fixed code state $\psi_0\psi_0^\dagger$; it does not alter the
corrected noisy code states.  Proposition~\ref{prop:direct-sum-dilation} applied to this Kraus family
then gives a dilation of $\mathcal R^*$ on a Kraus-index direct sum, without introducing a bipartite
quaternionic tensor product.  Hence (ii)$\Rightarrow$(i).
\end{proof}

The theorem is a finite-dimensional right-quaternionic specialization of established complex and
operator-algebraic error-correction mechanisms, but the descent is not a change of symbols.  The
doubled complex criterion initially permits complex coefficients $\lambda_{ab}$; the antiunitary
symmetry forces them into $Z(\HH)=\RR$, and the recovery must then be reconstructed from
right-$\HH$-linear syndrome isometries.  The two lemmas supply the fixed-sector and generation
steps needed for that descent.  Example~\ref{ex:kl-nondiagonal} shows why the real Gram matrix has
operational content: it records correlations and degeneracies among errors before the syndrome
basis is orthogonalized.

Uniqueness of $r$ follows because $P\ne0$.  Taking adjoints makes it real symmetric, compressing
$\sum_aE_a^\dagger E_a=I_N$ gives $\Tr r=1$, and diagonalization shows that
$\operatorname{rank}_{\RR}r$ is exactly the number of positive $d_\mu$ and hence of nonzero
orthogonalized syndrome sectors in the displayed recovery.

Equation~\eqref{eq:KLH} has a direct physical interpretation.  The matrix element
$r_{ab}$ depends on the pair of errors but not on the encoded state.  Therefore the environment or
syndrome can learn which error sector occurred without learning the logical information stored in
$\mathcal C$.  Positive semidefiniteness of $r$ is the Gram-matrix condition that permits the
error sectors to be orthogonalized.

The theorem can equivalently be read as an error-detection statement.  For $K\ge2$, a right-$\HH$-linear operator $F$ is \emph{detectable by the code} precisely when
its compression is central, $PFP=cP$ with $c\in\RR$, so that its action on the code carries no
logical-state dependence.  The restriction $c\in\RR$ is part of the definition and is not
automatic: by Example~\ref{ex:nonreal-compression} an operator can satisfy $PFP=cP$ with
$c\notin\RR$, and it then acts as a nontrivial logical operation.  Thus a family $\{E_a\}$ is correctable exactly when every cross
product $E_a^\dagger E_b$ is detectable in this sense.

Degenerate codes are included without an additional hypothesis.  If $r$ is singular, the
orthogonalized eigenvalues in Eq.~\eqref{eq:gram-diagonalization} may satisfy $d_\mu=0$.  The
corresponding error combination obeys $F_\mu P=0$ and creates no independent syndrome sector; only
the indices with $d_\mu>0$ enter the recovery.

\subsection{The one-dimensional boundary}

If $K=1$, the normalized state space of the logical quaternionic system consists of the single
matrix $[1]$.  A recovery can discard the noisy output and prepare that unique state, so exact
state correction carries no logical information and does not force Eq.~\eqref{eq:KLH}.  This is
not a counterexample for a quaternionic two-level system: a quaterbit is $\HH^2$ and has
$K=2$.  The same degeneracy occurs for a one-dimensional complex code, so this boundary is not
specific to $\HH$.  What is specific is where the proof breaks: $K=1$ is exactly the case excluded
by Lemma~\ref{lem:qherm-generation}, since
$\chi(\operatorname{Herm}_1(\HH))=\RR I_2$ generates only the scalar multiples of $I_2$ and not
$M_2(\CC)$.

\subsection{A non-diagonal Gram matrix and degenerate syndromes}

The theorem is basis-independent in the error labels, but the recovery construction diagonalizes
the Gram matrix.  The following example shows why that step is substantive: an exactly correctable
channel may be presented by a nonorthogonal Kraus family, and singular Gram matrices encode
degenerate error combinations without requiring a separate theorem.

\begin{example}[A non-diagonal error Gram matrix]\label{ex:kl-nondiagonal}
Let $\{e_1,e_2,e_3,e_4\}$ be the standard basis of $\HH^4$, set
$\mathcal C:=\operatorname{span}_{\HH}\{e_1,e_2\}$, and let
$P:=\operatorname{diag}(I_2,0_2)$.  Define
\[
 S:=\begin{pmatrix}0&I_2\\I_2&0\end{pmatrix},
 \qquad PSP=0,
 \qquad S^\dagger S=I_4.
\]
For $p\in[0,1]$, first introduce the orthogonal syndrome errors
$F_0:=\sqrt{1-p}\,I_4$ and $F_1:=\sqrt p\,S$, and then choose the equivalent Kraus family
\begin{equation}
 E_0:=\frac{F_0+F_1}{\sqrt2},
 \qquad
 E_1:=\frac{F_0-F_1}{\sqrt2}.
 \label{eq:mixed-kraus-family}
\end{equation}
Real orthogonal mixing preserves the channel and its trace-preserving normalization.  Direct
compression gives
\begin{equation}
 PE_a^\dagger E_bP=r_{ab}P,
 \qquad
 r=\frac12\begin{pmatrix}
 1&1-2p\\[1mm]
 1-2p&1
 \end{pmatrix}.
 \label{eq:nondiagonal-r}
\end{equation}
For $p\ne\tfrac12$, the Gram matrix is not diagonal even though the code is exactly correctable.
The orthogonal diagonalization in Eq.~\eqref{eq:gram-diagonalization} recovers the original
syndrome basis $F_0,F_1$ with eigenvalues $1-p$ and $p$.  At $p=0$ or $p=1$, $r$ is singular;
the zero-eigenvalue combination annihilates the code, illustrating how degeneracy is handled by the
same construction.  Thus the example separates two issues that a diagonal toy model would hide:
Kraus labels need not coincide with orthogonal syndrome labels, and a singular Gram matrix records
redundant error combinations rather than failure of correctability.
\end{example}

\section{Composition, discard, and recoverability after a subsystem model is chosen}
\label{sec:qcomposition}

Three constructions must be kept separate.  A Kraus-index direct sum dilates a single-system channel
without defining subsystems; Theorem~\ref{thm:KLH} corrects a subspace inside one right quaternionic
module; erasure, complementary channels, and locality require an actual composite together with a
discard map.  This section identifies the extra data needed for the third problem and shows how the
exact and approximate erasure criteria reduce to standard complex statements after a composite has
been selected.

\subsection{Why two right quaternionic modules do not determine a canonical composite}

A balanced tensor product over $\HH$ requires one factor to carry a compatible left action; two
right modules alone do not determine it.  If $M$ is a right module and $L$ a left module, balancing
identifies $(uq)\otimes v$ with $u\otimes(qv)$ for $u\in M$, $v\in L$, and $q\in\HH$.
Even after such an action is chosen, the frequently proposed simple-tensor inner product
\begin{equation}
 \langle u_1\otimes u_2,v_1\otimes v_2\rangle
 :=\langle u_1,v_1\rangle_{\HH}\langle u_2,v_2\rangle_{\HH}
 \label{eq:badinner}
\end{equation}
need not respect the balancing relation.  In $\HH\otimes_{\HH}\HH$ one has
$i\otimes1=1\otimes i$.  Pairing these representatives with $j\otimes1$ gives
\begin{equation}
 \langle i,j\rangle_{\HH}=-k,
 \qquad
 \langle1,j\rangle_{\HH}\langle i,1\rangle_{\HH}=j(-i)=k.
 \label{eq:tensorcounter}
\end{equation}
The disagreement shows that Eq.~\eqref{eq:badinner} is not well defined.

The realification control example has a canonical repair because each factor already carries the
distinguished operator $\mathsf J$.  The relation
\begin{equation}
 (\mathsf J_Ax)\otimes y\sim x\otimes(\mathsf J_By)
 \label{eq:real-composition-contrast}
\end{equation}
reconstructs complex scalar balancing.  By contrast, to impose
\begin{equation}
 (xq)\otimes y\sim x\otimes(qy)
 \label{eq:quaternionic-balancing-choice}
\end{equation}
one must first choose a compatible left-$\HH$ action on the second factor.  That choice, together
with an inner product and discard, is additional physical data rather than a consequence of the two
right-module structures.

Several consistent frameworks make different choices; they are collected in
Table~\ref{tab:composition-ledger}.
\begin{table*}[t]
\centering\small
\caption{Composition frameworks for quaternionic systems.  The rows are alternatives, not
notations for one canonical tensor product.}
\label{tab:composition-ledger}
\begin{tabular}{p{0.18\textwidth}p{0.32\textwidth}p{0.42\textwidth}}
\toprule
Framework & Additional structure & Operational consequence\\
\midrule
Razon--Horwitz construction & A real tensor product of quaternion algebras, quotient modules, and a selected scalar product & A many-body model exists, but the extra algebraic choices are part of the theory\\
Ordered circuits & A total order on quaternionic gates & Each ordered computation has a complex simulation; changing the order can change the output\\
Euclidean Jordan algebras & A monoidal category containing real, complex, and quaternionic observable algebras & Composites exist, but two quaternionic factors generally produce a real observable algebra\\
Complex symplectic model & An ordinary complex tensor product together with antiunitary symmetry constraints & Standard subsystem operations exist inside a constrained complex theory\\
\bottomrule
\end{tabular}
\end{table*}
These alternatives are discussed by Razon and Horwitz, Fern\'andez and Schneeberger, Graydon, and
Barnum--Graydon--Wilce \cite{Razon1991,Fernandez2004,Graydon2013,Barnum2020}.  A statement about a
quaternionic subsystem must therefore name the chosen framework and its discard map.

For positive integers $n,m,s$, write $Q_n:=M_n(\HH)_{\rm sa}$ for the quaternionic-Hermitian
matrices and $R_s:=M_s(\RR)_{\rm sa}$ for the real symmetric matrices.  Denote the universal
Euclidean Jordan composite by $\otimes_u$ and the standard embedded composite by $\boxtimes$.
Barnum--Graydon--Wilce obtain
$Q_n\otimes_uQ_m\simeq R_{4nm}$ for $n,m>2$ and
$Q_2\otimes_uQ_2\simeq R_{16}^{\oplus4}$ in their composition table
\cite{Barnum2020}.  In the standard embedded product,
$Q_n\boxtimes Q_m\simeq R_{4nm}$, including
$Q_2\boxtimes Q_2\simeq R_{16}$; the special $Q_2$ calculation is given in their
Appendix~B, especially Corollary~B.5.  The composite is therefore well specified and associative,
but it is not a quaternionic matrix algebra of the naive size.

\subsection{Erasure as a constant complementary channel}

Fix finite-dimensional complex Hilbert spaces $\mathcal H_L$, $\mathcal H_A$, and
$\mathcal H_B$ in one selected composite realization, and let
$V:\mathcal H_L\to\mathcal H_A\otimes\mathcal H_B$ be an isometric encoder.  For $X\in\End_{\CC}(\mathcal H_L)$ define
\begin{align}
 \mathcal N_A(X)&:=\Tr_B(VXV^\dagger),\notag\\
 \mathcal N_B(X)&:=\Tr_A(VXV^\dagger).
 \label{eq:erasure-channels}
\end{align}
Here $\Tr_A$ and $\Tr_B$ are the ordinary complex partial traces over the indicated factors.
These maps only exist after the tensor product and the discard operations have been fixed.  Write
$d_A:=\dim\mathcal H_A$, and let $I_A$, $I_B$, and $I_L$ denote the identity operators on
$\mathcal H_A$, $\mathcal H_B$, and $\mathcal H_L$, respectively.

\begin{proposition}[Exact erasure criterion in a chosen composite]\label{prop:chosen-composite-erasure}
Choose an orthonormal basis $\{|a\rangle\}_{a=1}^{d_A}$ of $\mathcal H_A$ and define
\begin{equation}
 E_a:=(\langle a|\otimes I_B)V:\mathcal H_L\to\mathcal H_B.
 \label{eq:erasure-block-kraus}
\end{equation}
The following are equivalent:
\begin{enumerate}[label=(\roman*),leftmargin=1.6em]
\item there is a CPTP map
$\mathcal R_B:\End_{\CC}(\mathcal H_B)\to\End_{\CC}(\mathcal H_L)$ such that
$\mathcal R_B\circ\mathcal N_B=\id_{\mathcal H_L}$;
\item there is a density matrix $\sigma_A$ on $\mathcal H_A$ such that
$\mathcal N_A(X)=\Tr_{\CC}(X)\sigma_A$ for every
$X\in\End_{\CC}(\mathcal H_L)$;
\item there is a density matrix $\sigma_A$ on $\mathcal H_A$ such that
\begin{equation}
 E_b^\dagger E_a=(\sigma_A)_{ab}I_L
 \qquad(1\le a,b\le d_A).
 \label{eq:erasure-block-KL}
\end{equation}
\end{enumerate}
\end{proposition}
\begin{proof}
Write $V=\sum_a|a\rangle\otimes E_a$.  Then
\begin{equation}
 \mathcal N_A(X)
 =\sum_{a,b}\Tr_{\CC}(XE_b^\dagger E_a)|a\rangle\langle b|.
 \label{eq:erasure-block-expansion}
\end{equation}
Equality with $\Tr_{\CC}(X)\sigma_A$ for every $X$ is equivalent, entry by entry, to
$\Tr_{\CC}[X(E_b^\dagger E_a-(\sigma_A)_{ab}I_L)]=0$ for every $X$.  Nondegeneracy of the
complex trace pairing gives Eq.~\eqref{eq:erasure-block-KL}; the converse follows by substitution,
so (ii) and (iii) are equivalent.

The retained channel has Kraus form
$\mathcal N_B(X)=\sum_aE_aXE_a^\dagger$.  Equation~\eqref{eq:erasure-block-KL} is precisely the
ordinary complex Knill--Laflamme condition for this channel on the whole logical input space, with
\begin{equation}
 \lambda_{ab}:=(\sigma_A)_{ba},
 \qquad \lambda=\sigma_A^{\mathsf T}.
 \label{eq:erasure-lambda-transpose}
\end{equation}
Because complex conjugation preserves positive semidefiniteness, $\lambda=\overline{\sigma_A}$ is positive semidefinite, and
$\Tr_{\CC}\lambda=1$.  The normalization also follows directly from
$V^\dagger V=\sum_aE_a^\dagger E_a=I_L$: summing the diagonal equations in
Eq.~\eqref{eq:erasure-block-KL} gives
$(\sum_a\lambda_{aa})I_L=I_L$.  The standard proof is unchanged for rectangular error maps: its necessity step uses the rank-one Choi matrix of the recovered identity channel, and its sufficiency step uses only the relations $E_a^\dagger E_b=\lambda_{ab}I_L$.  The rectangular form of the complex
Knill--Laflamme theorem therefore gives (i)$\Leftrightarrow$(iii)
\cite{KnillLaflamme1997}.  Combining the two equivalences proves the proposition.
\end{proof}

Proposition~\ref{prop:chosen-composite-erasure} makes the role of Section~\ref{sec:qec} explicit:
once a composite is chosen, erasure is governed by an ordinary Knill--Laflamme compression for the
block errors $E_a$.  At the exact level no separate information--disturbance theorem is required:
the block expansion identifies environmental constancy with the complex Knill--Laflamme equations,
whose recovery construction and necessity proof give both directions.  The quantitative
information--disturbance theorem is used only for the approximate statement in
Eq.~\eqref{eq:KSW-diamond}.  What is model dependent is not the algebraic criterion but the prior
choice of subsystems and discard.  The fixed state $\sigma_A$ need not be maximally mixed; for example,
$V|\psi\rangle=|0\rangle_A\otimes|\psi\rangle_B$ gives
$\mathcal N_A(\rho)=|0\rangle\langle0|$ for every $\rho$ and is perfectly correctable.

A useful no-go shows why imposing a maximally mixed erased marginal is too strong.  Let
$d_L:=\dim\mathcal H_L$, $d_A:=\dim\mathcal H_A$, and $d_B:=\dim\mathcal H_B$.  If $d_A=d_B$
and $\mathcal N_A(X)=\Tr_{\CC}(X)I_A/d_A$, then
\begin{equation}
 C_{\mathcal N_A}=I_L\otimes\frac{I_A}{d_A},
 \qquad
 \operatorname{rank}C_{\mathcal N_A}=d_Ld_A.
 \label{eq:maximally-mixed-replacement-choi}
\end{equation}
Any Stinespring realization with environment $\mathcal H_B$ has Choi rank at most $d_B$, so
$d_Ld_A\le d_B$ and hence $d_L\le d_B/d_A$.  In the balanced case $d_A=d_B$ this forces
$d_L\le1$.  A nontrivial code can have a constant erased state, but not
an obligatorily maximally mixed one in this balanced-dimension setting.

\subsection{Approximate recovery as approximate environmental constancy}

Let $\mathcal N:\End_{\CC}(\mathcal H_L)\to\End_{\CC}(\mathcal H_B)$ be the encoded noise channel
in the chosen complex realization, and let
$\widehat{\mathcal N}:\End_{\CC}(\mathcal H_L)\to\End_{\CC}(\mathcal H_A)$ be a complementary
channel obtained from a Stinespring dilation.  For a density matrix $\sigma$ on $\mathcal H_A$, let
$\mathcal S_\sigma(X):=\Tr_{\CC}(X)\sigma$.  The complex trace norm and diamond norm are
\begin{align}
 \|X\|_{1,\CC}&:=\Tr_{\CC}\sqrt{X^\dagger X},\notag\\
 \|\Lambda\|_{\diamond,\CC}
 &:=\sup_{r\ge1}\ \sup_{X\ne0}
 \frac{\|(\Lambda\otimes\id_r)(X)\|_{1,\CC}}{\|X\|_{1,\CC}}.
 \label{eq:complex-channel-norms}
\end{align}
In finite dimensions the reference dimension $r$ may be taken equal to the input dimension.  Define
\begin{align}
 \epsilon_{\rm rec}&:=\inf_{\mathcal R}
 \|\mathcal R\circ\mathcal N-\id_{\mathcal H_L}\|_{\diamond,\CC},\notag\\
 \delta_{\rm env}&:=\inf_{\sigma}
 \|\widehat{\mathcal N}-\mathcal S_\sigma\|_{\diamond,\CC}.
 \label{eq:eps-delta}
\end{align}
The first infimum is over CPTP recovery maps
$\mathcal R:\End_{\CC}(\mathcal H_B)\to\End_{\CC}(\mathcal H_L)$; the second is over density
matrices $\sigma$ on $\mathcal H_A$; and $\id_{\mathcal H_L}$ is the identity channel.  The
information--disturbance theorem gives
\begin{equation}
 \frac{\delta_{\rm env}^{\,2}}4
 \le\epsilon_{\rm rec}
 \le2\sqrt{\delta_{\rm env}}.
 \label{eq:KSW-diamond}
\end{equation}
Thus Proposition~\ref{prop:chosen-composite-erasure} is stable: exact correction corresponds to a
constant complementary channel, and approximate correction is controlled by the distance of the
entire complementary channel from the set of replacement channels
\cite{KSW2008,BenyOreshkov2010}.  The spectrum of one reduced state or the concurrence of one
codeword cannot establish this channel-level statement.

The logical dependency is now explicit.  Proposition~\ref{prop:direct-sum-dilation} realizes an
intrinsic channel without subsystems; Theorem~\ref{thm:KLH} corrects errors on one quaternionic
module; Proposition~\ref{prop:chosen-composite-erasure} and Eq.~\eqref{eq:KSW-diamond} apply only
after a composite and discard have been selected and represented complex-symplectically.  None of
these steps creates a canonical quaternionic partial trace.

\section{Operational choices beyond associativity}\label{sec:oct-fork}

The quaternionic analysis works because scalar multiplication and sequential operator composition
remain associative.  Octonions force a fork at precisely that point.  They underlie exceptional
structures such as $G_2$, the Albert algebra, and the Cayley plane and have repeatedly appeared in
generalized quantum models \cite{GunaydinPironRuegg1978,DeLeoAbdelKhalek1996,Baez2002}, but the
phrase ``octonionic quantum mechanics'' does not specify one operational theory.  The useful
question is instead: after associativity is lost, which definitions of states, probabilities,
dynamics, composites, and recovery remain available, and in which model?

The answer is organized as a decision tree.  Para-linear Hilbert theory keeps genuinely octonionic
amplitudes by modifying linearity and composition.  Cochain twisting moves nonassociativity into a
categorical associator.  Fixed quaternionic sectors restore the complete associative theory already
developed.  Jordan models keep octonionic state geometry in both spin-factor and exceptional cases, Clifford
envelopes keep associative operators, and Moufang registers keep the finite multiplication law in
complex labels.  The basic
associator calculation below explains why these routes cannot be merged without additional choices.

Let
\begin{equation}
 \OO:=\Span_{\RR}\{1,e_1,\ldots,e_7\}
 \label{eq:octonion-basis}
\end{equation}
be the real octonion division algebra.  For
$x=x_0+\sum_{i=1}^{7}x_i e_i$, define
\begin{align}
 \overline x&:=x_0-\sum_{i=1}^{7}x_i e_i,\notag\\
 |x|_{\OO}&:=\sqrt{x\overline x}
 =\left(x_0^2+\sum_{i=1}^{7}x_i^2\right)^{1/2}.
 \label{eq:octonion-norm}
\end{align}
The associated real Euclidean inner product is
$\langle x,y\rangle_{\RR}:=\RePart(\overline x y)$, so ``orthogonal'' and ``orthonormal'' in the
octonionic sections always refer to this real inner product and the scalar norm $|\cdot|_{\OO}$.
These are distinct from $|\cdot|_{\HH}$, $\|\cdot\|_{\mathcal H_{\HH}}$, and the complex
channel norms defined earlier.  The transition from quaternions is not obtained by replacing
$\HH$ with $\OO$ in the preceding formulas.  The obstruction is the failure of associativity,
which changes the relation between scalar multiplication and sequential operator composition.

For $a,b,c\in\OO$, define the associator
\begin{equation}
 [a,b,c]:=(ab)c-a(bc).
 \label{eq:oct-associator}
\end{equation}
The norm $|\cdot|_{\OO}$ is multiplicative because $\OO$ is a normed division algebra \cite{Baez2002,ConwaySmith2003}.  The octonions are alternative: the associator is alternating, and any two elements generate an
associative subalgebra \cite{Schafer1966,Baez2002,ConwaySmith2003}.  This explains why
quaternionic sectors are abundant, but it does not make a product of three arbitrary octonions
independent of bracketing.

For $a\in\OO$, let $L_a\in\End_{\RR}(\OO)$ be left multiplication,
$L_a(x):=ax$.  Since composition in $\End_{\RR}(\OO)$ is associative, it cannot coincide with
the nonassociative octonionic product on every triple.  The discrepancy is exact:
\begin{equation}
 L_aL_b-L_{ab}=-[a,b,\cdot\,].
 \label{eq:left-obstruction}
\end{equation}
Indeed, both sides applied to $x\in\OO$ equal $a(bx)-(ab)x$.  Consequently, there is no faithful
unital algebra homomorphism from $\OO$ into an associative algebra that preserves the full
octonionic product: associativity of the target would send every associator to zero.

This is a precise obstruction, not a blanket statement that octonionic Hilbert spaces or operators
do not exist.  It says only that ordinary associative operator composition cannot faithfully be
identified with global octonionic multiplication.  A valid octonionic model must therefore state
which product, composition law, positive cone, and notion of subsystem it uses.

For a concrete nonassociative calculation, fix the oriented Fano lines
\begin{equation}
 \begin{gathered}
 (1,2,4),\ (2,3,5),\ (3,4,6),\ (4,5,7),\\
 (5,6,1),\ (6,7,2),\ (7,1,3).
 \end{gathered}
 \label{eq:fano-lines}
\end{equation}
The convention means that cyclic products along each listed line are positive and reversing an
ordered pair changes the sign.  Then
\[
 (e_1e_2)e_3=e_4e_3=-e_6,
 \qquad
 e_1(e_2e_3)=e_1e_5=e_6,
\]
so
\begin{equation}
 [e_1,e_2,e_3]=-2e_6\ne0.
 \label{eq:explicit-octonion-associator}
\end{equation}
This one example already shows why a sequence of ordinary operators cannot be identified globally
with multiplication by a single octonion independently of bracketing.

\begin{table*}[t]
\centering\small
\caption{Operationally distinct octonionic models.  Each row has its own state space and dynamics.}
\label{tab:oct-models}
\begin{tabular}{@{}p{0.17\textwidth}p{0.24\textwidth}p{0.25\textwidth}p{0.24\textwidth}@{}}
\toprule
Model & States and measurements & Dynamics & Error-correction status\\
\midrule
Para-linear Hilbert model & Octonionic Hilbert modules and para-linear inner-product functionals & Para-linear operators with regular composition and corrected adjoint & An intrinsic CP/Choi/QEC theory is not yet supplied by the cited framework\\
Cochain-twist category & Positive and $\CPstar$ objects in a twisted graded Hilbert category & $\CPM$/$\CPstar$ morphisms and categorical dilations & Complete categorical semantics, but equivalent to ordinary complex graded theory\\
Fixed quaternionic sector & Quaternionic density matrices, effects, and POVMs in $S^N$ with $S\cong\HH$ & Sector-preserving quaternionic channels & The quaternionic Choi and Knill--Laflamme results apply\\
Jordan model ($H_2$, $H_3$) & Positive trace-one elements and effects in $H_2(\OO)$ or $H_3(\OO)$ & Jordan automorphisms or a separately chosen positive-map category & Exceptional-factor and higher-spin-factor categorical obstructions prevent an automatic global Kraus/QEC theory\\
Clifford/operator envelope & Ordinary real or complex density matrices on the algebra generated by $L_a$ & Ordinary associative CP maps & Standard QEC of the encoding, not intrinsic octonionic QEC\\
Finite Moufang register & Ordinary complex amplitudes labelled by $O_{16}$ & Reversible loop oracles and unitary transforms & Detects multiplication and bracketing defects in a complex register\\
\bottomrule
\end{tabular}
\end{table*}

The remaining subsections follow the rows of Table~\ref{tab:oct-models} as a sequence of operational
choices, not as competing notations for one unfinished formalism.  Para-linearity keeps genuine
$\OO$-valued amplitudes; cochain twisting relocates nonassociativity to categorical coherence; fixed
sectors restore associativity; Jordan models retain octonionic state geometry, with the spin-factor
and exceptional cases kept distinct; Clifford envelopes represent left multiplication by ordinary
operators; and Moufang registers isolate the discrete
multiplication law.  Terms such as ``octonionic state'' or ``octonionic channel'' are therefore
incomplete until one of these models has been selected.  The first route keeps octonionic amplitudes
most directly.

\subsection{Para-linear octonionic Hilbert and operator theory}\label{sec:paralinear}

The purpose of this model is to retain octonionic amplitudes and an octonionic-valued inner product
without pretending that ordinary linearity survives nonassociativity.  It supplies a legitimate
spectral and operator calculus that may serve as the analytic foundation for future dynamics; the
question examined here is exactly how far that foundation already reaches.

The ordinary representation obstruction does not prevent a nonassociative operator theory whose
notion of linearity is adapted to $\OO$.  Octonionic Hilbert spaces were introduced by Goldstine
and Horwitz \cite{GoldstineHorwitz1964,GoldstineHorwitz1966}.  Modern work develops the theory
through para-linearity, regular composition, and a corrected adjoint
\cite{HuoRen2022,HuoRen2023,HuoRenSabadini2025,HuoRenXu2026,HuoRenSabadiniXu2026}.

The module structure needed here is slightly more specific than an unspecified right module.
Let $M$ be a right $\OO$-module and let $M'$ be an $\OO$-bimodule.  Following the modern
bimodule theory, define the associative and central subspaces by
\begin{align}
 \mathcal A(M')&:=\{y\in M':(pq)y=p(qy)\text{ for all }p,q\in\OO\},\notag\\
 Z(M')&:=\{y\in M':py=yp\text{ for all }p\in\OO\}.
\end{align}
For an octonionic bimodule these spaces coincide,
$\RePart M'=\mathcal A(M')=Z(M')$ \cite{HuoRen2023}.  The canonical real-part projection
$\RePart_{M'}:M'\to\RePart M'$ is
\begin{equation}
 \RePart_{M'}(y)
 :=\frac{5}{12}y-\frac{1}{12}\sum_{i=1}^{7}e_iye_i.
 \label{eq:module-real-part}
\end{equation}
For a real-linear map $T:M\to M'$, a vector $x\in M$, and a scalar $p\in\OO$, define the
second right associator
\begin{equation}
 B_p(T,x):=T(x)p-T(xp).
 \label{eq:second-assoc}
\end{equation}
The map $T$ is right para-linear when
\begin{equation}
 \RePart_{M'}\!\left(B_p(T,x)\right)=0
 \qquad\text{for every }x\in M\text{ and }p\in\OO.
 \label{eq:paralinearity}
\end{equation}
This is the module-valued definition used in the modern operator theory
\cite{HuoRen2023,HuoRenSabadini2025}.  Strict right-linearity would require
$B_p(T,x)=0$; para-linearity retains only the vanishing of its projected real part.  For
$M'=\OO$, Eq.~\eqref{eq:module-real-part} reduces to the ordinary scalar real part, and the
condition is the one entering the octonionic Riesz representation theorem.

The operator theory does more than define a class of maps.  If $M''$ is an $\OO$-bimodule,
$S:M'\to M''$ is para-linear, and $T:M\to M'$ is real-linear, their regular composition
$S\Rcomp T$ is the unique para-linear map whose projected real part agrees with that of
$S\circ T$.  On bounded para-linear endomorphisms of a Hilbert $\OO$-bimodule, regular
composition and the corrected adjoint give an octonionic involutive Banach algebra in the
nonassociative sense of \cite{HuoRen2023,HuoRenSabadini2025}; regular composition is not
associative in general.  Spectral decomposition, partial isometries, and functional
calculi can then be formulated in that setting.  These results evade
Eq.~\eqref{eq:left-obstruction} by changing the composition law; they do not claim that ordinary
composition equals octonionic multiplication.

There is nevertheless an operational boundary.  Selinger's $\CPM$ construction starts from a
dagger compact category, whose morphism composition is associative \cite{Selinger2007}.  The
para-linear category under regular composition is nonassociative, so ordinary $\CPM$ cannot be
applied directly.  Passing to an associative real-linear envelope changes the morphism theory.  A
complete intrinsic quantum-channel model would still need a positive cone stable under the chosen
composition, complete positivity, trace preservation, a Choi object, a composite, and a discard
map.  The cited para-linear results provide a substantial operator foundation, but not yet all of
those operational structures.
Para-linearity preserves genuinely octonionic scalar values, but its regular composition is
nonassociative and an intrinsic CPTP category has not yet been derived from the cited framework.  A
second route changes the categorical associator instead of weakening linearity.

\subsection{The cochain-twist categorical model}\label{sec:cochain-model}

This route addresses a different goal: retain a fully compositional quantum-process calculus while
placing the octonionic associator in the monoidal structure rather than in the scalar field.  It is
therefore useful for deciding whether nonassociative multiplication can coexist with categorical
complete positivity and, equally importantly, whether the resulting theory is operationally new.

A different octonionic route keeps ordinary complex Hilbert spaces but modifies the monoidal
associator and braiding.  Let $G:=(\mathbb Z_2)^3$, written additively, and let
$\{u_g:g\in G\}$ be a homogeneous basis.  A normalized cochain is a function
$F:G\times G\to\{\pm1\}$ satisfying $F(0,g)=F(g,0)=1$ for every $g\in G$.
Choose $F$ so that, for all $g,h\in G$,
\begin{equation}
 u_g\cdot_Fu_h:=F(g,h)u_{g+h}
 \label{eq:twisted-oct-product}
\end{equation}
reproduces the complexified octonion multiplication.  The homogeneous labels $g\in G$ are not
identified term-by-term with the cyclic Fano indices in Eq.~\eqref{eq:fano-lines}; an explicit
coordinate identification differs only by a permutation of the seven nonzero labels.  For
$g,h,\ell\in G$, its coboundary
\begin{equation}
 \phi(g,h,\ell):=
 \frac{F(h,\ell)F(g,h+\ell)}{F(g+h,\ell)F(g,h)}
 \label{eq:cochain-associator}
\end{equation}
is the categorical associator, and
$R(g,h):=F(g,h)F(h,g)^{-1}$ gives the symmetry
\cite{AlbuquerqueMajid1999}.

Let $\mathcal C_F$ be the category of finite-dimensional $G$-graded complex Hilbert spaces with
degree-preserving linear maps, associator $\phi$, and symmetry $R$.  Let
$\mathcal C_0:=\mathrm{FdHilb}_G$ denote the same graded Hilbert spaces with the ordinary
associator and flip.  A \emph{dagger category} has an involutive contravariant operation
$f\mapsto f^\dagger$ on morphisms.  A \emph{strong dagger monoidal functor} carries tensor
products through invertible tensorators that are unitary; it is symmetric when those tensorators
also intertwine the two symmetries.  The symbols $\CPM(\mathcal C)$ and
$\CPstar[\mathcal C]$ denote the standard categorical constructions that pass from a dagger
category $\mathcal C$ to, respectively, completely positive processes and finite-dimensional
abstract quantum--classical systems.  The cochain realization of the octonions is due to Albuquerque and Majid
\cite{AlbuquerqueMajid1999}; more generally, cochain-twisted gauge theories are known to be
twisting-equivalent to their untwisted counterparts \cite{Majid2005}.  The point of the next
result is to state the finite-dimensional quantum-process consequence explicitly.  The underlying
monoidal untwisting is standard; what is checked here is that the tensorator is unitary and
dagger-compatible and therefore transports the $\CPM$ and $\CPstar$ operational constructions.

\begin{corollary}[Dagger operational consequence of cochain untwisting]\label{cor:dagger-untwist}
The identity-on-objects functor $\mathcal U_F:\mathcal C_F\to\mathcal C_0$ with tensorator
$\mu_{V,W}:V\otimes_0W\to V\otimes_FW$ given by
\begin{equation}
 \mu_{V,W}(v_g\otimes w_h):=F(g,h)^{-1}v_g\otimes w_h
 \label{eq:unitary-tensorator}
\end{equation}
for homogeneous vectors $v_g\in V_g$ and $w_h\in W_h$ is a strong unitary dagger symmetric
monoidal equivalence.  Consequently,
\begin{equation}
 \CPM(\mathcal C_F)\simeq\CPM(\mathcal C_0),
 \qquad
 \CPstar[\mathcal C_F]\simeq\CPstar[\mathcal C_0].
 \label{eq:cpm-cpstar-equivalence}
\end{equation}
The twisted octonion algebra object is carried to the ordinary group algebra $\CC[G]$.
\end{corollary}
\begin{proof}
For homogeneous vectors, $\mu_{V,W}$ multiplies each degree-$(g,h)$ summand by the phase
$F(g,h)^{-1}$.  Hence $\mu_{V,W}^{\dagger}=\mu_{V,W}^{-1}=\mu_{V,W}$ because
$F(g,h)\in\{\pm1\}$.  Substituting homogeneous vectors of degrees $g,h,\ell$ into the
monoidal-functor pentagon reduces its two paths to the scalar identity
\[
 \phi(g,h,\ell)
 =\frac{F(h,\ell)F(g,h+\ell)}
        {F(g+h,\ell)F(g,h)},
\]
which is exactly Eq.~\eqref{eq:cochain-associator}.  The symmetry square is the identity, and the
symmetry compatibility diagram reduces to
$R(g,h)=F(g,h)F(h,g)^{-1}$.  Thus $\mathcal U_F$ is a strong unitary dagger symmetric
monoidal equivalence.

Dagger duals, evaluations, and coevaluations transport through a unitary dagger monoidal
equivalence, so $\mathcal C_F$ is dagger compact whenever $\mathcal C_0$ is.  Selinger's
$\CPM$ construction and the normalisable $\CPstar$ construction are functorial under such an
equivalence \cite{Selinger2007,CoeckeHeunenKissinger2014}, which gives
Eq.~\eqref{eq:cpm-cpstar-equivalence}.

Finally, let $m_F(u_g\otimes u_h)=F(g,h)u_{g+h}$ be the twisted multiplication.  The
multiplication transported to $\mathcal C_0$ is
\[
 m_0:=\mathcal U_F(m_F)\circ\mu_{\CC[G],\CC[G]},
 \qquad
 m_0(u_g\otimes u_h)=u_{g+h}.
\]
Hence the distinguished twisted algebra object becomes the ordinary group algebra $\CC[G]$.
\end{proof}

This result has both a positive and a limiting interpretation.  The twisted category has
well-defined categorical positive maps, completely positive maps, Choi objects, tensor products,
discard maps, and categorical dilations.  However, because $G$ is finite abelian,
\begin{equation}
 \CC[G]\cong\CC^8
 \label{eq:group-algebra-c8}
\end{equation}
as a commutative finite-dimensional $*$-algebra.  The model is therefore dagger-equivalent to ordinary complex $G$-graded quantum theory; it is not
an inequivalent theory obtained by using $\OO$ as a new scalar field, and it does not define
intrinsic para-linear complete positivity.  A different route is to retain genuinely hypercomplex
dynamics by restricting all operational data to an associative quaternionic subalgebra of $\OO$.

\subsection{Associative quaternionic sectors inside the octonions}\label{sec:oct-sectors}

Fixed quaternionic sectors are the most immediately operational octonionic regime.  They identify
the largest associative ``islands'' on which states, POVMs, channels, Choi matrices, and exact error
correction are already available without modification.  Their limitation is equally precise: the
noise and the recovery must preserve the chosen sector.

Let $\Imag\OO:=\{x\in\OO:\overline x=-x\}$ be the imaginary subspace.  If
$u,v\in\Imag\OO$ satisfy
$\langle u,u\rangle_{\RR}=\langle v,v\rangle_{\RR}=1$ and
$\langle u,v\rangle_{\RR}=0$, then
\begin{equation}
 S_{u,v}:=\Span_{\RR}\{1,u,v,uv\}\cong\HH
 \label{eq:quaternionic-sector}
\end{equation}
is an associative quaternionic subalgebra.  The seven oriented lines of a fixed Fano diagram give
coordinate examples, while all quaternionic sectors form the homogeneous space $G_2/SO(4)$
\cite{HarveyLawson1982,Baez2002}.

The existence and geometry of quaternionic subalgebras of $\OO$ are standard
\cite{HarveyLawson1982,Baez2002}.  A fixed sector inherits the entire quaternionic operational
theory because all amplitudes, operators, and errors remain inside an associative copy of $\HH$.
The corollary below makes this inheritance precise by transporting the state, channel, Choi, and
error-correction results through a fixed algebra isomorphism $S\cong\HH$.

\begin{corollary}[Sector-preserving reduction]\label{cor:sector-reduction}
Fix a quaternionic subalgebra $S\subset\OO$ and a unital $*$-algebra isomorphism
$\iota:S\to\HH$.  Let $N\in\mathbb N$, let
$\mathcal C\subset S^N$ be a right-$S$ submodule of $S$-dimension at least two, let
$P\in M_N(S)$ be its orthogonal projector.  Let $m\in\mathbb N$ and let
$E_1,\ldots,E_m\in M_N(S)$ satisfy
$\sum_{a=1}^{m}E_a^\dagger E_a=I_N$.  Exact correction by a
sector-preserving channel is equivalent to the existence of a real positive-semidefinite matrix
$r=(r_{ab})\in M_m(\RR)$ satisfying
\begin{equation}
 PE_a^\dagger E_bP=r_{ab}P
 \qquad(1\le a,b\le m).
 \label{eq:sector-KL}
\end{equation}
The state embedding, Choi criterion, recovery, and direct-sum dilation also transport through
$\iota$.
\end{corollary}
\begin{proof}
Applying $\iota$ entrywise identifies $S^N$ with $\HH^N$ and $M_N(S)$ with $M_N(\HH)$ while
preserving multiplication, conjugation, adjoints, positivity, and the real trace.
Proposition~\ref{prop:symplectic-image} and Proposition~\ref{prop:state-embedding} transport the
fixed operator and state sectors, Theorem~\ref{thm:choi-quaternionic} transports the Choi criterion,
and Theorem~\ref{thm:KLH} transports Eq.~\eqref{eq:sector-KL} together with its recovery.  The
direct-sum dilation is transported entrywise in the same way.
\end{proof}

If $g\in G_2:=\operatorname{Aut}(\OO)$, then $gS$ is another quaternionic sector and the
transported code, projector, and errors satisfy the same equation with the same matrix $r$.  This
is covariance of the sectorwise theorem.  It does not cover noise that moves amplitudes through
several incompatible sectors; such a process depends on associators and is not a binary
quaternionic channel.

Sector restriction gives a complete channel and QEC theory, but only after a fixed copy of $\HH$
has been selected.  Jordan models take the opposite route: they preserve octonionic state geometry
without assuming an underlying associative Kraus algebra.

\subsection{Jordan state spaces: spin factors and composite boundaries}\label{sec:jordan-models}

The Jordan route keeps the convex geometry of octonionic states and effects rather than an
octonionic amplitude module.  Its purpose is to provide a well-defined one-system probabilistic
model and then ask how far it can be composed with other systems.  This makes the section
operationally relevant even though it does not produce a Kraus theory by itself.

A second intrinsic octonionic model does not use $\OO^N$ as a Hilbert module.  For
$n\in\{2,3\}$, let $H_n(\OO)$ denote the Hermitian $n\times n$ octonionic matrices.  For $x,y\in H_n(\OO)$, define the
Jordan product
\begin{equation}
 x\circ y:=\tfrac12(xy+yx).
 \label{eq:jordan-product}
\end{equation}
With this product, these spaces are Euclidean Jordan algebras.  The product is defined by the displayed
symmetrization; it is not ordinary associative matrix multiplication.  The two cases are not on the
same footing.  The algebra $H_2(\OO)$ is isomorphic to the ten-dimensional spin factor
$J\mathrm{Spin}_9\cong V_9$ and is therefore \emph{special}; only the Albert algebra
$H_3(\OO)$ is exceptional \cite{Baez2002}.  These cases meet different categorical boundaries.
The direct-summand obstruction below concerns the exceptional Albert algebra.  The spin factor
$V_9$ avoids that particular obstruction, but the dagger-compact framework of
Barnum--Graydon--Wilce cannot be enlarged to include higher spin factors while retaining compact
closure and the representation of states as morphisms \cite{Barnum2020}.

The positive cone is
\begin{equation}
 H_n(\OO)_+:=\{x\circ x:x\in H_n(\OO)\}.
 \label{eq:jordan-positive-cone}
\end{equation}
A normalized state is an element $\rho\in H_n(\OO)_+$ with
$\Tr\rho=1$, where $\Tr$ is the real sum of the diagonal entries.  An effect is an element $e\in H_n(\OO)$ satisfying $0\le e\le I_n$, where $I_n$ is the
$n\times n$ identity matrix, and the probability rule is
\begin{equation}
 p(e\mid\rho):=\Tr(e\circ\rho).
 \label{eq:jordan-born}
\end{equation}
Primitive idempotents of $H_2(\OO)$ give $\OP^1\cong S^8$, while those of the Albert algebra
$H_3(\OO)$ give the Cayley plane $\OP^2$ \cite{GunaydinPironRuegg1978,Baez2002}.

This layer supplies states, effects, probabilities, and reversible Jordan automorphisms.  It does
not by itself choose a tensor product, a Kraus representation, a complementary channel, or a
Knill--Laflamme theorem.  Under the non-signalling and categorical composite axioms studied by
Barnum, Graydon, and Wilce, an exceptional Jordan algebra cannot occur as a nonclassical
composite factor except against an essentially classical system \cite{Barnum2020}.  Their
categorical analysis gives a separate restriction for $H_2(\OO)\cong V_9$: higher spin factors
cannot be added to the real--complex--quaternionic dagger-compact framework while preserving its
state-as-morphism interpretation.  These are imported boundary results.  A different composite
proposal is not ruled out in advance, but its axioms and operational maps must be supplied
explicitly.

Jordan models provide positive cones and reversible symmetries, yet their composite restrictions do
not automatically supply a global Kraus or QEC formalism.  An associative operator envelope solves a
different problem: it represents octonionic left multiplication on an ordinary real space while
remembering exactly where multiplication ceases to be faithful.

\subsection{The associative Clifford envelope}\label{sec:clifford-envelope}

The Clifford envelope has a concrete implementation purpose: represent octonionic left
multiplications by ordinary real or complex matrices so that they can be simulated, composed, and
subjected to standard channel or error-correction methods.  What is measured in this encoding is
not the scalar field itself but the discrepancy between sequential operators and a single octonionic
product, namely the associator defect.

Ordinary real-linear operators retain a useful associative shadow of octonionic multiplication.
Choose an orthonormal imaginary basis $e_1,\ldots,e_7$ of $\Imag\OO$ and let
$L_{e_i}\in\End_{\RR}(\OO)$ be left multiplication by $e_i$.  Alternativity implies
\begin{equation}
 L_{e_i}L_{e_j}+L_{e_j}L_{e_i}
 =-2\delta_{ij}I_8
 \qquad(1\le i,j\le7),
 \label{eq:clifford-envelope}
\end{equation}
where $\delta_{ij}$ is the Kronecker delta and $I_8$ is the identity on the underlying
$8$-dimensional real vector space.  These are the defining relations of the real Clifford algebra
$\Cl_{0,7}$, so the left multiplications give an $8$-dimensional real spinor representation
\cite{LawsonMichelsohn1989,DepiesSmithAshburn2023}.  Since $\Cl_{0,7}\cong
M_8(\RR)\oplus M_8(\RR)$ and $\OO$ is an irreducible module of real type, the associative
subalgebra of $\End_{\RR}(\OO)$ generated by $L_{e_1},\ldots,L_{e_7}$ is the image of one summand,
namely all of $M_8(\RR)$.  The envelope is therefore an ordinary real matrix algebra, which is
precisely why channels defined on it are channels of an encoding.

For $a,b\in\OO$, define the associator-defect operator
\begin{equation}
 A_{a,b}:=[a,b,\cdot\,]\in\End_{\RR}(\OO).
 \label{eq:defect-operator}
\end{equation}
Then Eq.~\eqref{eq:left-obstruction} becomes
\begin{equation}
 L_aL_b=L_{ab}-A_{a,b}.
 \label{eq:envelope-defect}
\end{equation}
An ordinary real or complex channel may act on the associative algebra generated by the
$L_{e_i}$.  Such a channel is a channel of an encoding.  It is not an intrinsic octonionic Kraus
map, because sequential composition $L_aL_b$ differs from multiplication by the single octonion
$ab$ whenever $A_{a,b}\ne0$.

The Clifford envelope is operationally conventional but changes the scalar theory into a real
operator encoding.  The final route keeps only the finite multiplication table and studies it as
reversible data in an ordinary complex register.

\subsection{Finite Moufang loops and quasialgebraic models}\label{sec:moufang-models}

The finite-loop route isolates the discrete nonassociative multiplication table.  It is useful for
reversible multiplication oracles, exact verification of Moufang identities, and algorithms that
compare different bracketings without introducing ambiguous octonionic amplitudes.

The signed octonion basis
\begin{equation}
 O_{16}:=\{\pm1,\pm e_1,\ldots,\pm e_7\}
 \label{eq:o16}
\end{equation}
is closed under octonionic multiplication and forms a finite Moufang loop.  Let
$\mathcal H_{16}:=\CC^{16}$ with orthonormal basis
$\{|g\rangle:g\in O_{16}\}$.  This is an ordinary complex Hilbert space whose labels carry a
nonassociative multiplication table.  For labels $g,h,k\in O_{16}$, reversible multiplication oracles and coherent comparisons
of $(gh)k$ with $g(hk)$ are therefore legitimate quantum circuits, but the amplitudes are complex,
not octonionic.

The corresponding loop algebra is nonassociative and can be studied using quasialgebras,
quasimatrices, and Hopf quasigroups \cite{AlbuquerqueMajid1999,KlimMajid2010}.  These structures
supply algebraic analogues of group multiplication, coproducts, and antipodes.  They do not by
themselves determine a positive cone, a discard map, or a completely positive channel category.
Those operational ingredients must either be added separately or obtained from the cochain-twist
model of Section~\ref{sec:cochain-model}.

The finite-loop and Clifford-envelope routes are useful precisely because their operational meaning
is explicit.  The former is a complex register with nonassociative labels; the latter is an
associative operator encoding of left multiplication.  Neither should be identified with a global
Hilbert space over $\OO$.

\section{Conclusion and open problems}\label{sec:conclusion}

The real and quaternionic constructions exhibit the same general lesson in two different forms.
Complex quantum mechanics can be written on a real Hilbert space, but the orthogonal complex
structure $\mathsf J$, its commutant, and the balanced composite remain part of the operational
theory.  Quaternionic quantum mechanics can be represented on a doubled complex Hilbert space, but
the antiunitary symplectic structure $\Theta$, its fixed operator algebra, and the normalization
$\rho\mapsto\chi(\rho)/2$ remain essential.  Exact simulation therefore does not imply that the
encoded scalar structure has been removed.

Within the quaternionic fixed sector, the Choi criterion gives a representation-independent test for
whether a complex CPTP map admits quaternionic Kraus operators.  Its proof is a symplectic
specialization of the fixed-real-form Choi mechanism known from imaginarity theory.  The
right-quaternionic Knill--Laflamme theorem goes further into the noncommutative structure: the
doubled complex condition is shown to descend to an intrinsic criterion with coefficients in the
real center, including the one-dimensional state-space boundary, degenerate error families, and an explicit
right-$\HH$-linear recovery.  The surrounding Hilbert-module, trace, POVM, complex-simulation,
Stinespring, and information--disturbance ingredients are established results used in a common
finite-dimensional convention.

For octonions, the conclusion is a controlled separation rather than a single process theory.
Para-linear Hilbert spaces retain octonionic amplitudes but presently lack an intrinsic CPTP and QEC
package.  Cochain twisting has complete categorical semantics but, by Corollary~\ref{cor:dagger-untwist}, untwists to ordinary complex
graded quantum theory.  Fixed quaternionic sectors inherit the full quaternionic theory by
Corollary~\ref{cor:sector-reduction}.  Jordan,
Clifford-envelope, and Moufang-register models retain, respectively, octonionic state geometry,
associative operators, and the finite multiplication law, but they answer different operational
questions.  The comparison prevents these models from being conflated and identifies the data that
must be added before a global octonionic channel or recovery theorem can be stated.

Three concrete problems remain: defining complete positivity, discard, Choi objects, and
Stinespring dilation intrinsically for para-linear morphisms; constructing controlled dynamics that
move coherently among quaternionic sectors; and formulating recovery criteria for ternary bracketing
defects or multi-time nonassociative processes.

\appendix
\section{Auxiliary symplectic and antiunitary details}\label{app:symplectic-details}

This appendix records two short facts used in the state and Choi arguments.  They are consequences
of standard quaternionic spectral theory, but including them makes the normalization and
antiunitary steps independent of an external calculation.

\subsection{Spectral doubling and reflection of positivity}

Let $A=A^\dagger\in M_N(\HH)$.  The quaternionic spectral theorem gives
$A=UDU^\dagger$, where $U\in\Sp(N)$ and
$D=\operatorname{diag}(\lambda_1,\ldots,\lambda_N)$ with
$\lambda_a\in\RR$.  Since $D$ is real,
\begin{equation}
 \chi(D)=\begin{pmatrix}D&0\\0&D\end{pmatrix},
 \qquad
 \chi(A)=\chi(U)\chi(D)\chi(U)^\dagger.
 \label{eq:appendix-spectral-doubling}
\end{equation}
Thus every quaternionic eigenvalue occurs twice in the complex representation.  It follows
immediately that
\begin{align}
 A\ge0
 &\quad\Longleftrightarrow\quad
 \chi(A)\ge0,
 \notag\\
 \operatorname{rank}_{\CC}\chi(A)
 &=2\operatorname{rank}_{\HH}A.
 \label{eq:appendix-rank-positivity}
\end{align}
Together with Eq.~\eqref{eq:trace-doubling}, this also proves the converse direction of
Proposition~\ref{prop:state-embedding}: a positive $\Theta$-fixed complex density matrix is
$\chi(A)$ for a unique positive quaternionic $A$, and trace one forces
$\operatorname{Tr}_{\HH}A=1/2$.

\subsection{Fixed bases for an antiunitary involution}

Let $\mathcal E$ be a finite-dimensional complex Hilbert space and let
$C:\mathcal E\to\mathcal E$ be antiunitary with $C^2=I$.  For any $v\in\mathcal E$, set
\begin{equation}
 x:=\frac{v+Cv}{2},
 \qquad
 y:=\frac{v-Cv}{2i}.
 \label{eq:appendix-real-decomposition}
\end{equation}
Antilinearity gives $Cx=x$ and $Cy=y$, while $v=x+iy$.  Hence the fixed set
$\mathcal E^C:=\{w:Cw=w\}$ is a real Hilbert space whose complexification is all of
$\mathcal E$.  Real Gram--Schmidt therefore produces an orthonormal $C$-fixed basis of every
$C$-stable subspace.  Applied to the eigenspaces of the positive Choi matrix with
$C=\mathfrak C$, this is the fixed-eigenbasis step used in
Theorem~\ref{thm:choi-quaternionic}.

\section*{Data and code availability}
No experimental data are used.

\end{document}